\documentclass[12pt,reqno]{amsart}

\usepackage{xcolor}
\usepackage{soul}

\usepackage{amsmath}
\usepackage{amssymb}
\usepackage{hhline}
\usepackage{graphicx}
\newtheorem{theorem}{Theorem}

\newtheorem{theoremd}{Theorem}

\newtheorem{examp}[theoremd]{Example\!}

\newtheorem{rk}[theoremd]{Remark\!\!}

\newcommand\bib[1]{\bibitem[#1]{#1}}

\renewcommand\a{\alpha}

\renewcommand\d{\delta}

\newcommand\fg{\mathfrak{g}}
\newcommand\g{\gamma}

\renewcommand\H{{\mathcal H}}

\renewcommand\O{\Omega}
\newcommand\oo{\omega}
\newcommand\op[1]{\mathop{\rm #1}\nolimits}
\newcommand\ot{\otimes}
\newcommand\p{\partial}

\newcommand\R{{\mathbb R}}
\renewcommand\t{\times}

\newcommand\vp{\varphi}
\newcommand\we{\wedge}
\newcommand\x{\xi}
\newcommand\z{\theta}

\begin{document}
 \title[On integrability of rank 2 SR-structures]{On integrability of certain \\ rank 2 sub-Riemannian structures}
 \author[B.\ Kruglikov, A.\ Vollmer, G.\ Lukes-Gerakopoulos]{Boris Kruglikov, Andreas Vollmer, \\ Georgios Lukes-Gerakopoulos}
% \author[B.\ Kruglikov, G.\ Lukes-Gerakopoulos, A.\ Vollmer]{Boris Kruglikov, Georgios Lukes-Gerakopoulos, \\ Andreas Vollmer}
 \address{BK: \ Institute of Mathematics and Statistics, University of Troms\o, Troms\o\ 90-37, Norway.
\quad E-mail: {\tt boris.kruglikov@uit.no}. \newline
 \hphantom{W} AV: \ Mathematisches Institut, Friedrich-Schiller-Universit\"at, 07737 Jena, Germany.
 \quad Email: {\tt andreasdvollmer@gmail.com}. \newline
% \quad Email: {\tt andreas.vollmer@uni-jena.de}. \newline
 \hphantom{W} AV: \ INdAM - Politecnico di Torino, Dipartimento di Scienze Matematiche,
 Corso Duca degli Abruzzi 24, 10129 Torino, Italy.
%\quad Email: {\tt andreas.d.vollmer@gmail.com}.
 \newline
 \hphantom{W} GLG: \ Institute of Theoretical Physics, Faculty of Mathematics and Physics,
Charles University in Prague, 121 16 Prague, Czech Republic.
 \quad Email: {\tt gglukes@gmail.com}.}
 \date{}
 %\thanks{BK and AV were supported by the NFR and DAAD cooperation grant 2014-2015 respectively. AV is a research fellow of the Istituto Nazionale di Alta Matematica and received financial support from GRK 1523 (DFG). GLG was supported by the UNCE-204020 grant.}
 \thanks{BK and AV were supported by the NFR and DAAD cooperation grant 2014-2015 respectively. AV is a research fellow of Istituto Nazionale di Alta Matematica and thanks GRK 1523 (DFG) and the project FIR-2013 Geometria delle equazioni differenziali for financial support. GLG was supported by the UNCE-204020 grant.}
 \maketitle

 \begin{abstract}
We discuss rank 2 sub-Riemannian structures on low-dimensional manifolds
and prove that some of these structures in dimension 6, 7 and 8 have a maximal amount of
symmetry but no integrals polynomial in momenta of low degrees, except for those coming from
the Killing vector fields and the Hamiltonian, thus indicating non-integrability of the
corresponding geodesic flows.
 \end{abstract}

%%%%%%%%%%%%%%%%%
\section*{Introduction}

A sub-Riemannian (SR) structure on a (connected smooth) manifold $M$ consists of a
completely non-holonomic (or bracket-generating) vector distribution $\Delta\subset TM$
and a Riemannian metric $g\in\Gamma(S^2_+\Delta^*)$ on it.
For points $x,y\in M$ denote by $\H(x,y)$ the space of integral (horizontal) curves
$\g:[0,1]\to M$, $\dot\g\in\Delta$, joining $x$ to $y$: $\g(0)=x$, $\g(1)=y$.
It is nonempty by the Rashevsky-Chow theorem.

The length functional $l_g(\g)=\int_0^1\|\dot\g\|_gdt$ on the space of horizontal curves
defines the {\em sub-Riemannian\/} distance on $M$ by
 $$
d_g(x,y)=\inf\limits_{\g\in\H(x,y)}l_g(\g).
 $$
A curve $\g\in\H$ is called {\em geodesic} if it locally minimizes the length
between any two (close) points with respect to $d_g$. The description of most
geodesics (normal ones) is given by the Euler-Lagrange variational
principle. There is a Hamiltonian reformulation of this principle, called
the Pontrjagin maximum principle \cite{PMP}. It allows one to consider
the sub-Riemannian geodesic flow as the usual Hamiltonian flow on $T^*M$ (abnormal extremals
play no role in this respect and will be ignored in this paper).
We will recall this together with the other relevant material in Section \ref{S1}.

As in the standard theory of Riemannian geodesics, the metric $g$ is {\em integrable} if this Hamiltonian
flow is integrable on $T^*M$ in the Liouville sense, i.e.\ there are almost everywhere
functionally independent integrals $I_1=H,I_2,\dots,I_n$ that Poisson-commute $\{I_k,I_l\}=0$
(see \cite{A,AKN} and also \cite{BF} for a review of methods and problems).

In this paper we investigate certain aspects of  integrability of SR-structures on
vector distributions of rank 2 (the smallest rank in non-holonomic mechanics). In general,
SR-structures need not be integrable. For the first time, this was illustrated with a precise example
in \cite{MSS} by Montgomery, Shapiro and Stolin. More examples can be found in \cite{Kr}.
We will focus on left-invariant SR-structures on Carnot groups, which serve as tangent cones
(nilpotent approximations) for general SR-structures. In Riemannian geometry, the tangent
cone is the Euclidean space and it is integrable. This integrability does not carry over to the
sub-Riemannian case.

We discuss {\em integrability}\footnote{We consider the integrals that are analytic
in momenta. For a quadratic Hamiltonian $H$, the existence of such an integral
implies by \cite{Dar,Wh} the existence of an integral that is homogeneous polynomial in momenta.
Moreover, in all our cases we need only one additional integral $I$ commuting with $H$ and the
linear integrals, so this $I$ can always be assumed homogeneous polynomial in momenta.}
of SR-structures and particularly pose the specific question whether it
is related to the amount of {\em symmetry} present in these structures. On Carnot
groups of dimensions up to 5 the geodesic flow of all left-invariant SR-structures are Liouville integrable
(see Section \ref{S2}), however starting from dimension 6 we show that the
final polynomial integrals, required for Liouville integrability,
cease to exist at least in low degrees (up to 6), even in the maximally symmetric situations. For precise formulations in dimension 6, 7 and 8, see Theorems \ref{Thm6D}, \ref{Thm7D} and \ref{Thm8D}, respectively.

In Section~\ref{S8}, we reduce the corresponding systems of PDEs to systems with 2 degrees of freedom in
a convenient form that allows us to consider obstructions for integrability in a uniform setting.
The reduced systems provide a parametric 3-components first order system of ODEs. Its dynamics is interesting in its own right (we speculate that the case corresponding to dimension 6 is similar to a forced pendulum).

In Section~\ref{S9}, we complement our results with the trajectory portraits that demonstrate
irregular dynamics. Our computations show that the systems exhibit chaotic behavior for various
values of parameters in the reduced formulation, providing more evidence of non-integrability.
In dimension 8 our study agrees with the numerical observations of \cite{Sa}.

The combination of established low-degree non-integrability, the reduced formulation
(the known integrable quadratic Hamiltonians with 2 degrees of freedom have integrals of $\op{deg}\le4$),
and the numerical evidence strongly suggests that generic SR-structures
are in general not Liouville integrable with analytic in momenta integrals.
 % even if they have the maximum amount of symmetry.
In the Riemannian setting this was recently proved in \cite{KM2}.

The technique we use in sections~\ref{S5}, \ref{S6} and \ref{S7} is inherited from the work \cite{KM1},
where it was exploited to prove rigorously non-existence of low-degree integrals for the Zipoy-Voorhees metric from general relativity (for related work on this topic, see \cite{LG, MPS, V1}). We will explain the method in detail in Sections~\ref{S3} and \ref{S4}. In short, it allows us to reduce the search of integrals of
a fixed degree $d$ to a linear algebra problem, namely to a computation of the rank of a matrix with
the size polynomially growing with $d$. The entries of this matrix are integers, and the computer
verification, solely based on evaluation of the rank, gives a rigorous proof of the result.
To the best knowledge of the authors, it is at present the only method
that allows one to make non-existence statements for the class of integrals under consideration.

\medskip

\textsc{Acknowledgment.} We are grateful to the referees for useful suggestions improving the
exposition. One of them proposed an idea of modular computations that we implemented in the revision.
We also thank Vladimir Matveev for encouragement throughout our work.

%%%%%%%%%%%%%%%%%
\section{Hamiltonian systems and Sub-Riemannian structures}\label{S1}

Let us recall some basic facts from the theory of integrable systems and
sub-Riemannian geometry.

\subsection{\sl Integrable Hamiltonian flows.}\label{S1.1}
Let $(W^{2n},\O)$ be a symplectic manifold. A Hamiltonian vector field is
the $\O$-dual to an exact 1-form $dH$: $\O(X_H,\cdot)=dH$.
The Poisson bracket is then $\{F,G\}=\O(X_F,X_G)$.

The Hamiltonian system is called Liouville integrable if
in addition to $I_1=H$ there are integrals $I_2,\dots,I_n$, which are in involution,
$\{I_j,I_k\}=0$, and which are functionally independent almost everywhere. By the Liouville
theorem \cite{A} a full measure set $W'\subset W$ is then foliated by
cylinders (tori in the compact case), and each cylinder has a neighborhood
with coordinates $\phi\in \mathbb{T}^{n-r}\t\R^r$, $J=J(I)\in\R^n$ such that
$\O=dJ\we d\phi$, $\{I=\op{const}\}\simeq \mathbb{T}^{n-r}\t\R^r$ and the
flow is $\dot\phi=\Phi(I)$ for some function $\Phi$. The noncompactness rank $r$ can
be positive for sub-Riemannian geodesic flows, even on compact manifolds.

It might happen that in addition to involutive integrals $I_1,\dots,I_n$ there
are some other integrals $I_{n+1},\dots,I_{n+k}$. Then the motion
is restricted to sub-cylinders in $\mathbb{T}^{n-r}\t\R^r$ --
this is the resonance (the non-compact case is more subtle).
Existence of additional integrals (more than $n$ commuting ones) is usually referred to
as super-integrability.

Let $(M,g)$ be a Riemannian manifold. The geodesic flow
$\vp_t:TM\to TM$ is the Hamiltonian flow on
$T^*M\stackrel{\sharp^g}\simeq TM$ with the standard
symplectic structure $\O$ if we choose the Hamiltonian $H=\frac12\|p\|^2$,
$p\in T^*M$. The metric $g$ is called integrable if the flow $\vp_t$ is
Liouville-integrable. If $M$ is compact, then the noncompactness rank $r=0$.

\subsection{\sl Pontrjagin maximum principle.}\label{S1.2}

Consider now the non-holonomic case. We start with an arbitrary completely
non-holonomic distribution $\Delta\subset TM$ and a smooth field of Riemannian metrics
$g_x\in S^2_+\Delta^*_x$, $x\in M$.

Let $(T^*M,\O)$ be the cotangent bundle equipped with the standard
symplectic structure. The sub-Riemannian metric $g$ yields the isomorphism
$\sharp^g:\Delta^*\to\Delta$.  Consider the inclusion $i:\Delta\hookrightarrow TM$.
The following composition defines a vector bundle morphism $\Psi_g$:
 $$
T^*M\stackrel{i^*}\to\Delta^*\stackrel{\sharp^g}\to\Delta\stackrel{i}\to TM.
 $$
Contrary to the Riemannian situation this is not an isomorphism. Indeed, we have:
$\op{Ker}(\Psi_g)=\op{Ann}\Delta$ and $\op{CoKer}(\Psi_g)=TM/\Delta$.

Define the Hamiltonian function $H$ on $T^*M$ as the composition
 $$
T^*M\stackrel{i^*}\to\Delta^*\stackrel{\sharp_g}\to\Delta
\stackrel{\frac12\|\cdot\|^2_g}\to\R.
 $$
This function can be locally described as follows. Let $\x_1,\dots,\x_k$
be some orthonormal basis of vector fields tangent to $\Delta$. Every vector
field is a fiber-linear function on $T^*M$. So we have
$H=\frac12\sum_1^k\x_i^2$.

The Pontrjagin maximum principle \cite{PMP} states that trajectories of
the Hamiltonian vector field $X_H$ in the region $\{H>0\}$ projected to $M$ are extremals of
the corresponding variational problem. They are called (normal) geodesics.

 \begin{examp}\label{HSR}
Consider the Heisenberg group $G=\op{Heis}_3$ with the standard left-invariant metric on
$\Delta=\langle \p_{x_1}+x_2\p_{x_3},\p_{x_2}\rangle\subset T\R^3(x_1,x_2,x_3)$. Solving the Hamiltonian equation
for $2H=(p_1+x_2p_3)^2+p_2^2$ we see that SR-geodesics are spirals in the direction of the $x_3$-axis,
projecting to arbitrary (including radius $\infty$) circles on the plane $\R^2(x_1,x_2)$.
 \end{examp}

\subsection{\sl Vector distributions.}\label{S1.3}
Given a vector distribution $\Delta\subset TM$ we define its weak derived flag by bracketing the
generating vector fields: $\Delta_1=\Delta$, $\Delta_{i+1}=[\Delta,\Delta_i]$.
The distribution is non-holonomic if $\Delta\subsetneq\Delta_2$ and completely non-holonomic if
$\Delta_k=TM$ for some $k$. We will assume that the rank of the distributions $\Delta_i$ is constant throughout $M$,
then $(\dim\Delta_1,\dim\Delta_2,\dots,\dim\Delta_k)$ is called the growth vector of $\Delta$.

The family of graded vector spaces $\{\mathfrak{g}_i=\Delta_i/\Delta_{i-1}\}$,
equipped with the natural bracket induced by the commutators of vector fields,
form a sheaf of graded nilpotent Lie algebras $\fg=\fg_1\oplus\dots\oplus\fg_k$ over $M$.
In this paper we consider only the strongly regular case, when it is a bundle
(i.e.\ the structure constants in the fiber do not depend on $x\in M$).
The typical fiber is then called the Carnot algebra of $\Delta$.

For the rank 2 distribution $\Delta\subset TM$ the prolongation is defined as follows \cite{AK,Mon}.
Let $\hat M=\mathbb{P}\Delta=\{(x,\ell):x\in M,\ell\subset\Delta_x\}$ be the natural
$\mathbb{S}^1$-bundle over $M$ with projection $\pi:\hat M\to M$. Then the prolonged distribution
$\hat\Delta\subset T\hat M$ is given by $\hat\Delta_{x,\ell}=\pi_*^{-1}(\ell)\subset T_{x,\ell}\hat M$.

 \begin{examp}\label{HSRcont}
The prolongation of the tangent bundle of $\R^2$ is $(\op{Heis}_3,\Delta)$.
% the next prolongation is the Engel structure discussed below.
The prolongation of $(\op{Heis}_3,\Delta)$ is the Engel structure discussed below.
 \end{examp}

Even though the SR-behavior can be quite different, the prolonged distribution has the geometry
readable off the original distribution and, starting from dimension 5, we will assume that $\Delta$ is
not a prolongation of a rank 2 distribution from lower dimensions. % (= not de-prolongable)

%%%%%%%%%%%%%%%%%
\section{SR-structures on Carnot groups of dimension 3 to 5}\label{S2}

In this section we discuss left-invariant SR structures on low-dimen\-sio\-nal Carnot groups.
We will show that up to dimension 5 all of them are Liouville integrable. This holds for distributions
$\Delta$ of all ranks, but since the concern of the paper is $\op{rank}(\Delta)=2$, we restrict to this case.
 % These results are easy to obtain but we are unaware that they have been published elsewhere. In addition, this section fixes the notations used in the rest of the paper.

A Carnot group $G$ is a graded nilpotent Lie group, with its Lie algebra
$\fg=\fg_1\oplus\dots\oplus\fg_k$ being bracket-generated by $\fg_1$,
and distribution $\Delta\subset TG$ corresponding to it.
Equipped with a left-invariant Riemannian metric on $\Delta$, such a group naturally serves
as a tangent cone at a chosen point of a general SR-structure, see e.g. \cite{BR} for details.

In what follows we use the following notations. A basis $e_i$ of $\fg$ generates the basis $\oo_i\in(\fg^*)^*$
of linear functions on $\fg^*$, given by
 $$
\oo_i(p)=\langle p,e_i\rangle,\quad p\in\fg^*.
 $$
We identify $e_i$ with the left-invariant vector fields on $G$, and $\oo_i$ with the left-invariant linear
functions on $T^*G$. Their right-invariant analogs will be denoted by $f_i$ and $\theta_i$ respectively.

The Lie-Poisson structure $\nabla_{LP}$ on $\fg^*$ induces the Poisson structure
$(\nabla_{LP},-\nabla_{LP})$ on $\fg^*\oplus\fg^*$ and this yields the following commutation relation
of the above functions with respect to the canonical symplectic structure on $T^*G$:
If $[e_i,e_j]=c_{ij}^ke_k$, then
 $$
\{\oo_i,\oo_j\}=c_{ij}^k\oo_k,\ \{\oo_i,\theta_j\}=0,\ \{\theta_i,\theta_j\}=-c_{ij}^k\theta_k.
 $$

% It is easy to see that in each of dimensions 3, 4, 5 there is a unique Carnot algebra with
% $\dim\fg_1=2$ that is not a prolongation from lower dimensions, and
% a unique SR-structure (up to an automorphism).

\subsection{\sl Dimension~3: the Heisenberg SR-structure}\label{S2.1}

In dimension 3 the only Carnot group\footnote{Left-invariant SR-structures on 3D Lie groups
are considered in Appendix \ref{A:1}.} is $G=\op{Heis}_3$.
The Carnot algebra is $\mathfrak{heis}_3=\fg_1\oplus\fg_2$ with $\fg_1=\langle e_1,e_2\rangle$, $\fg_2=\langle e_3\rangle$ and the only relation $[e_1,e_2]=e_3$.

The Hamiltonian $H=\frac12(\omega_1^2+\omega_2^2)$ has two integrals:
$I_2=\theta_1$ and the Casimir $I_3=\theta_3=\omega_3$. In the coordinates of Example
from \S\ref{S1}\!\ref{S1.2}
we have: $I_2=p_1$, $I_3=p_3$. These $I_1=H,I_2,I_3$ are involutive and functionally independent,
and so yield Liouville integrability.

There is also a fourth (noncommuting) integral $I_4=\theta_2=p_2+x_1p_3$ confining the motion
to the cylinders $\mathbb{S}^1\times\R^1\subset T^*G=G\times\fg^*$, and making the system
super-integrable.

Actually, for all systems considered in this paper whenever we establish Liouville integrability,
the super-integrability (but not maximal super-integrability) will follow. Indeed, we will always
indicate a right-invariant linear form (commuting with the left-invariant Hamiltonian)
that is functionally independent of the other integrals.

\subsection{\sl Dimension~4: the Engel SR-structure.}\label{S2.2}

In dimension 4 we also have only one SR-structure, related to the well-known Engel structure.

The graded nilpotent Lie algebra is $\fg=\fg_1\oplus\fg_2\oplus\fg_3=\langle e_1,e_2\rangle
\oplus\langle e_3\rangle\oplus\langle e_4\rangle$ with the nontrivial commutators:
 $$
[e_1,e_2]=e_3,\ [e_1,e_3]=e_4.
 $$
The Hamiltonian is $H=\frac12(\omega_1^2+\omega_2^2)$, and $I_2=\z_2$, $I_3=\z_3$, $I_4=\z_4$
together with $I_1=H$ form a complete set of integrals. Adding $I_5=\z_1$ makes the Hamiltonian system
super-integrable (notice though that $I_5$ does not commute with $I_2,I_3$).

In coordinates on $G$ we have\footnote{This and similar formulae are obtained via realization
of the basis $e_i$ as left-invariant vector fields on $G$. For the Engel structure: $e_1=-(\p_{x_1}+x_2\p_{x_3}+x_3\p_{x_4})$,
$e_2=\p_{x_2}$, $e_3=\p_{x_3}$, $e_4=\p_{x_4}$. The right invariant vector fields $f_i$
are such fields on $G$ that commute with $e_j$ and have the same values at the unity of $G$.}:
 $$
2H=(p_1+x_2p_3+x_3p_4)^2+p_2^2,
 $$
and the integrals are:
\begin{align*}
  I_2	&=\theta_2=p_2+x_1p_3+\tfrac12x_1^2p_4, &  I_4 &=\theta_4=p_4;\quad\\
  I_3	&=\theta_3=p_3+x_1p_4, &                  (I_5 &=-\theta_1=p_1).
\end{align*}
Alternatively, to get an involutive set of integrals, we can use the integrals
$J_2=I_5$, $J_3=I_4$ and the Casimir function $J_4=\oo_3^2-2\oo_2\oo_4$:
 $$
J_2=p_1,\ J_3=p_4,\ J_4=p_3^2-2p_2p_4=I_3^2-2I_2I_4.
 $$

The obtained integrals establish Liouville integrability of $H$.

\subsection{\sl Dimension~5: the Cartan SR-structure.}\label{S2.3}

In dimension 5 there are two SR-structures: one on the prolongation of the Engel structure
(a partial case of the Goursat structure, easily seen to be integrable, so we skip it)
and the other related to Cartan's famous $(2,3,5)$
distribution. The Carnot algebra is the positive part of the grading, corresponding to the
first parabolic subalgebra of the exceptional Lie algebra $\op{Lie}(G_2)$:
$\fg=\fg_1\oplus\fg_2\oplus\fg_3=\langle e_1,e_2\rangle
\oplus\langle e_3\rangle\oplus\langle e_4,e_5\rangle$. The nontrivial commutators of $\fg$ are:
 \begin{equation}\label{CartanD}
[e_1,e_2]=e_3,\ [e_1,e_3]=e_4, \ [e_2,e_3]=e_5.
 \end{equation}
The Hamiltonian is $H=\frac12(\omega_1^2+\omega_2^2)$. In terms of right-invariant vector fields and the corresponding
linear functions, we have the following involutive set of integrals:
the Casimir function
 $$
I_2=\z_1\z_5-\z_2\z_4+\tfrac12\z_3^2=\omega_1\omega_5-\omega_2\omega_4+\tfrac12\omega_3^2,
 $$
together with the linear integrals $I_3=\z_3,\ I_4=\z_4,\ I_5=\z_5$.
Again adding $I_6=\z_1$ makes the Hamiltonian system super-integrable
(the next obvious candidate $I_6'=\z_2$ is already functionally dependent
with the previous integrals; they do not commute with $I_2$).

In coordinates on $G$ we have:
 $$
2H=(p_1-\tfrac12x_2p_3-x_1x_2p_4)^2+(p_2+\tfrac12x_1p_3+x_1x_2p_5)^2,
 $$
and with the notation $J_\pm=x_1p_4\pm x_2p_5$ the integrals are:
\begin{align*}
  I_2	&=p_1p_5-p_2p_4+\tfrac12p_3^2+\tfrac12J_{-}^2+\tfrac12p_3J_{+}, & I_4 &=p_4,\\
  I_3	&=p_3, & I_5 &=p_5.
\end{align*}
The additional integral is either
$I_6=p_1+\frac12x_2p_3+(x_3-\frac12x_1x_2)p_4+\frac12x_2^2p_5$ or
$I_6'=p_2-\frac12x_1p_3-\frac12x_1^2p_4+(x_3+\frac12x_1x_2)p_5$.

%%%%%%%%%%%%%%%%%
\section{Discussion: on detecting integrability}\label{S3}

As discussed in the previous section, all left-invariant SR-structures on Carnot groups of dimensions
$<6$ are Liouville integrable, and 6 is the smallest nontrivial dimension from this viewpoint.
Reference~\cite{MSS} provides an example of a non-integrable left-invariant SR-structures on a
Carnot group in 6D. Namely, it is supported on a rank 3 distribution with growth vector $(3,5,6)$
on the flag manifold $\op{SL}(4)/B$, where $B$ is the Borel subgroup. In the respective $|3|$-grading
$\mathfrak{sl}(4)=\sum_{i=-3}^3\mathfrak{g}_i$ with $\mathfrak{b}=\sum_{i=0}^3\mathfrak{g}_i$
the distribution $\Delta$ corresponds to $\mathfrak{g}_{-1}$.
 % the symmetry algebra of the distribution is $\mathfrak{sl}_4$, while that of SR is $\fg_{-}\oplus\R$.
Integrability in loc.cit.\ is understood in algebraic sense, and the obstruction to it
is based on a result of Ziglin (using the monodromy group).

We study SR-structures on rank 2 distributions of other 6D Carnot groups $G$, and especially
ask whether integrability is linked to the maximal symmetry of a structure.
Of course, every symmetry field yields, by the Noether theorem, an integral linear in momenta
(to be called, in what follows, a Noether integral), but the question is if these
(plentiful in our cases) linear integrals suffice for integrability.

Here, we consider two kinds of infinitesimal symmetries. On the one hand, we consider vector
fields whose flow preserves the underlying rank 2 distribution of the SR-structure. On the other hand,
we consider Killing symmetries whose flow, in addition, preserves the SR-metric. In both cases we observe
that the existence of an additional symmetry appears not to be linked to the existence of a complete set of
involutive integrals. In each of our examples, considering a system with $D=\dim(G)$ degrees of freedom,
we will be able to identify $D-1$ involutive integrals (1 Hamiltonian and $D-2$ Noether integrals),
henceforth concentrating on the search for the {\em final} integral.
 % (this could be called {\em last} integral, by analogy with the last multiplier,
 % but this contrasts with the notion of the first integral)

In dimension $D=6$, we will list all left-invariant SR-structures on Carnot groups and discuss their
integrability. It will be shown that a final integral of low degree does not exists even if $\Delta$ is
the maximally symmetric distribution; on the other hand for the maximally symmetric SR-structure such
an integral does exist. In dimension $D=7$ or 8 we will focus on maximally-symmetric SR-structures.
These have $D+1$ linear independent Killing symmetries (the corresponding Noether integrals are
not in involution). This follows from the Tanaka theory that we review in Appendix \ref{A:2}.
We can also straightforwardly find the Noether integrals by realizing the group law via the
Baker-Campbell-Hausdorff formula and solving the relevant system of PDEs.
 % Higher order polynomial integrals can also be found directly given their degree $d$.

Non-integrability, on the other hand, is more difficult to prove.
There are few methods to detect it, depending on the integrability setup.
Analytic non-integrability on a compact manifold follows from positivity of the topological entropy,
see \cite{T} (this is guaranteed if a horseshoe subsystem exists,
which can be obtained via scattering data, see \cite{B}), but this cannot be applied in our case.

Obstructions for algebraic integrability can be found via differential Galois theory \cite{MR},
and the Painlev\'e test. The latter assumes the integrals to be rational in all coordinates.
In contrast, we are interested in first integrals that are smooth in the base coordinates and
polynomial in momenta (such integrals, also known as Killing tensors, have applications in
mathematical physics and general relativity), and the above tests are not applicable.
The method that can detect existence of this type of integrals was
proposed in the work \cite{KM1}. Before describing it in details in the next section,
let us explain the simple idea behind it.

The condition governing existence of an integral of degree $d$ is an overdetermined
system of $\binom{d+D}{D-1}$ linear differential equations on $\binom{d+D-1}{D-1}$ unknown
functions of $D$ variables. It is of finite type, meaning the system is reducible to ODEs.

Checking the explicit compatibility conditions can be cumbersome. Instead, we
compute all differential consequences, i.e.\ take partial derivatives of the PDEs in the system,
and it is enough to take those of total order $\le d+1$, cf.\ \cite{Wo,KM1}.
At this stage, the problem of solving differential equations is reduced to linear algebra: kernel
of the matrix of the system, evaluated at
a fixed point, corresponds to degree $d$ integrals. Some of those are products of apriori known linear
and quadratic integrals, which gives the lower bound on the nullity of the matrix.

The decision on existence of the final integral is thus reduced to the computation of the rank of this
matrix. If the nullity is the minimal possible (decided by the rank), no additional integral exists.

Similar to applications of the Galois theory or Painlev\'e test, our method can be implemented on
a computer. Only integer numbers are involved, so the symbolic computations are exact, containing no
roundings or approximations (with large numbers involved, these computations are more reliable
than a calculation by a human).
Significance of such proofs in mathematics has steadily increased in recent years.

In the next section, we give the technical details of the method, and then use it to prove
Theorems~\ref{Thm6D}, \ref{Thm7D} and \ref{Thm8D} in a mathematically rigorous,
computer-assisted manner.
The reader interested in independent verification of
our computations is invited to download the Maple-code from the supplement to
arXiv:1507.03082v2.

%%%%%%%%%%%%%%%%%
\section{Method to check existence of the final integral}\label{S4}

Similarly to \cite{KM1}, we show that the systems, considered further in this paper,
lack one final integral $F$ required for Liouville
integrability. In all our cases, with $D=\dim G$ degrees of freedom, we have $D-2$ commuting
Noether integrals that we normalize to be $p_3,\dots,p_D$.

This reduces the system to a $(D-2)$-parametric Hamiltonian system with $2$ degrees of freedoms
(simply let $p_i=\op{const}_i$). It means that the unknowns (=coefficients of $F$) depend on 2
variables only.

We however do not perform the reduction until Section \ref{S8},
and keep the momenta $p=(p_1,\dots,p_D)$ in the general form.
The Hamiltonian thus takes the form $H=H(x_1,x_2,p_1,\dots,p_D)$.

A first integral that is smooth by the base variables $x=(x_1,x_2)$ and polynomial in momenta of
degree~$d$, and that commutes with the Noether integrals $p_3,\dots,p_D$, has the form

 \begin{equation}\label{eqn:integral-F}
F=\sum_{|\tau|=d}a_\tau(x_1,x_2)\,p^\tau.
 \end{equation}
($p^\tau=\prod_{i=1}^Dp_i^{\tau_i}$ for a multi-index $\tau=(\tau_1,\dots,\tau_D)$ and $|\tau|=\sum_{i=1}^D\tau_i$).
The requirement that functions of the form~$F$ are integrals is given by the Poisson bracket relation $\{H,F\}=0$, which is a homogeneous polynomial in momenta of degree $d+1$.
Thus we have $\binom{d+D}{D-1}$ first order PDE (coefficients of $\{H,F\}$) on
$\binom{d+D-1}{D-1}$ unknowns (coefficients $a_\tau$ of $F$).
Let us call this linear PDE system $S_d$.

\subsection{\sl The bounds on the number of integrals.}\label{S4.1}

Instead of treating the differential system $S_d$, we consider the associated system of linear equations,
given by fixing a point $o\in G$. Namely, denote by $S_d^{(k)}$ the $k$-th prolongation of $S_d$.
This is the system obtained by differentiating the PDEs from $S_d$ by $x_1,x_2$ up of total order
$\le k$. The total number of the equations is consequently
$m_{d,k}=\binom{d+D}{D-1}\!\cdot\!\binom{k+2}2$.

The unknowns are now the derivatives $a^o_{\tau;\sigma}=\partial_\sigma a_\tau(o)$ (jets)
with multi-indices $\tau=(\tau_1,\dots,\tau_D)$ and $\sigma=(\sigma_1,\sigma_2)$
(of different size) of lengths $|\tau|=d$ and $|\sigma|\le k$. The collection of the unknowns
is denoted $V=V_d^{(k)}$, and represented as a column-vector. Their number, denoted $\#V$
(the height of this vector), is equal to $n_{d,k}=\binom{d+D-1}{D-1}\!\cdot\!\binom{k+3}2$.

The system $S_d^{(k)}$ evaluated at $o\in G$ has the form $M\cdot V=0$
with some $m_{d,k}\times n_{d,k}$ matrix $M=M_d^{(k)}$.
Let $\Lambda_d$ be the number of linearly independent first integrals of degree $d$.
We obviously have the upper bound, in which the right hand side stabilizes for $k=d+1$
(cf.\ \cite{Wo}):
 \begin{equation}\label{eqn:upper-bound}
\Lambda_d\le\d_{d}^{(k)}:=\# V_d^{(k)}-\op{rank}M_d^{(k)}.
 \end{equation}
Due to this stabilization let us denote in what follows:
$\d_d=\d_{d}^{(d+1)}$, $V_d=V_d^{(d+1)}$ and $M_d=M_d^{(d+1)}$.

On the other hand, our system possesses integrals $I_1=H$ (quadratic: $d_1=1$) and
$I_2=p_3$, \dots, $I_{D-1}=p_D$ (linear: $d_i=1$, $1<i<D$). The derived integrals
$\prod I_i^{m_i}$ of degree $\sum m_id_i=d$ will be called {\em trivial}. Thus we deduce the
lower bound
 \begin{equation}\label{Lambda}
\Lambda_d\ge\Lambda_d^0:=\sum_{i=0}^{[d/2]}\binom{d-2i+D-3}{D-3}.
 \end{equation}
If we show that the bounds in \eqref{eqn:upper-bound} and \eqref{Lambda} coincide,
$\Lambda_d^0=\d_{d}$, then we conclude that $\Lambda_d=\Lambda_d^0$,
and so all integrals of degree $d$ are trivial (reduced to already established Noether integrals).
This gives non-existence of the final integral in degree $d$.

\subsection{\sl The procedure.}\label{S4.2}

There are two important differences to \cite{KM1} that facilitate our computation.
First, our model is homogeneous, so the choice of point is not essential
(in the general case we have to choose a generic point, which gives the stable values of $\d_{d,k}$).
We always choose $o=(0,0)$ in the plane $\R^2(x_1,x_2)$.
Second, our Hamiltonian $H$ (scaled by some integer factor that is at most 288)
is a polynomial with integer coefficients, so we do not need to handle rational expressions.

Most complications are related to the calculation of $\op{rank}(M)$. The large size of the matrix makes
Gaussian elimination costly (the dimension of the configuration space in our computations
$D=6,7,8$ exceeds $\dim=4$ of \cite{KM1}).
But the matrix contains many zeros, and further simplifications are possible as follows:

1. In all equations all coefficients are kept integers at all stages by multiplying by their common
denominator. Proportional equations (rows of $M$) are removed.

2. We have the freedom to add constant multiples of trivial integrals to $F$ defined in~\eqref{eqn:integral-F}. At the point $o$ the corresponding unknowns can therefore be cleared from the equations.
Let $V_\text{spfl}\subset V$ embrace all (superfluent) unknowns that do not appear in the equations.
We remove the corresponding zero columns from M. We have: $\# V_\text{spfl}\geq\Lambda_d^0$.

3. Perform a partial solution of the system, iteratively solving the \textit{monomial} and
\textit{bimonomial} equations until no more such equations remain.
Let $V_\text{mon}$ and $V_\text{bimon}$ be the corresponding unknowns.

Denote the matrix of the reduced system (obtained from $M_d$ by the above simplifications) by
$M_\text{red}$, and let $V_\text{red}=V_d\setminus(V_\text{spfl}\cup V_\text{mon}\cup V_\text{bimon})$.
The reduced system then reads $M_\text{red}\cdot V_\text{red}=0$, and we get the formula
 \begin{multline*}
\d_d=\# V_d-\op{rk}(M_d)=(\# V_\text{red}+\# V_\text{mon}+\# V_\text{bimon}+\# V_\text{spfl})\\
-(\op{rk}(M_\text{red})+\# V_\text{mon}+\# V_\text{bimon})
=\# V_\text{red}+\# V_\text{spfl}-\op{rk}(M_\text{red}).
 \end{multline*}
Notice that $\# V_\text{red}=\op{rk}(M_\text{red})$ and $\# V_\text{spfl}=\Lambda_d^0$
imply $\delta_d=\Lambda_d^0$.

\subsection{\sl The modular approach}\label{S4.3}

Our algorithm confirms non-existence of the final integral of degree $d$ when
$\d_{d}=\Lambda_d^0$. The right hand side is given by \eqref{Lambda}, while
the left hand side depends on $\op{rank}M_d$ as in \eqref{eqn:upper-bound}.

To reduce the rank computation we work {\em modulo} $p$ for a prime $p$.
Denote by $L[p]$ the matrix obtained from a matrix $L$ by passing to mod $p$ entries.
Since for a square matrix $L$ we have $\det(L[p])=(\det L)\,\op{mod}\,p$, we conclude that
 \begin{equation}\label{modp}
\op{rank}M_d[p]\leq \op{rank}M_d.
 \end{equation}
For some specific values of $p$ the rank can actually decrease upon computing modular,
but for sufficiently large primes $p$ we have equality in \eqref{modp}.
Thus, if for some prime $p$
 $$
\d_d[p]:=\# V_d-\op{rank}M_d[p]=\Lambda_d^0,
 $$
then we conclude non-existence of the final integral of degree $d$.

The main complication is however to find such $p$. Our experiments show that the decisive $p$
grows fast with $D$. So even though these modular computations
are cheaper for every particular $p$, going successively by increasing primes actually increases
the computation time. Yet, a choice of a random increasing sequence of $p$ turns out to be useful.

% We however decided to make a random choice of $p$ (in an increasing sequence). This turned out very useful.

%%%%%%%%%%%%%%%%%
\section{Left-invariant SR-structures in dimension 6}\label{S5}

In this section we show a certain type of non-integrability for a rank~2 left-invariant distribution on a
6D Carnot group $G$. Every such 2-distribution $\Delta$ is encoded as the space $\fg_{-1}$
in the corresponding graded nilpotent Lie algebra $\fg$.

In 6D the growth vector is $(2,3,5,6)$ (recall we assumed that $\Delta$ is not a prolongation
of another rank 2 distribution),
and every such Lie algebra $\fg$ is a central 1D extension of the Cartan algebra from
Section \ref{S2}\!\ref{S2.3}, the distribution also being an integrable extension \cite{AK}.

Thus $\fg=\fg_1\oplus\fg_2\oplus\fg_3\oplus\fg_4=\langle e_1,e_2\rangle
\oplus\langle e_3\rangle\oplus\langle e_4,e_5\rangle\oplus\langle e_6\rangle$
has first commutators as in \eqref{CartanD}, which should be accompanied by
the brackets $\fg_1\otimes\fg_3\to\fg_4$. This leads to precisely three algebras,
called elliptic, parabolic and hyperbolic\footnote{The (2,3,5,6)-distributions are given by
a conformal quadric on $\fg_1$ due to conformal identification $\op{ad}_{\fg_2}:\fg_1\simeq\fg_3$,
whence elliptic, parabolic and hyperbolic.} in \cite{AK}.
We will study them in turn.

\subsection{\sl Integrability of the maximally symmetric elliptic SR-structure}\label{S5.1}

The elliptic (2,3,5,6)-distribution has the following structure equations:
 $$
[e_1,e_2]=e_3,\ [e_1,e_3]=e_4, \ [e_2,e_3]=e_5,\ [e_1,e_4]=e_6,\ [e_2,e_5]=e_6.
 $$
Its symmetry algebra has dimension 8 \cite{AK}, and it is not maximally symmetric as
a 2-distribution, but it supports the maximally symmetric SR-structure. Namely, defining the SR structure
by the orthonormal frame $e_1,e_2$, we conclude that its symmetry dimension is 7
(see Appendix \ref{A:2}). The corresponding Hamiltonian is
 $$
2H=\bigl(p_1-\tfrac12x_2p_3-x_1x_2p_4-\tfrac12x_1^2x_2p_6\bigr)^2+
\bigl(p_2+\tfrac12x_1p_3+x_1x_2p_5+\tfrac12x_1x_2^2p_6\bigr)^2.
 $$
There are two Casimir functions $I_6=\omega_6$ and $C=\frac12(\omega_4^2+\omega_5^2)-\omega_3\omega_6$.
For the maximally symmetric Hamiltonian $H=\frac12(\omega_1^2+\omega_2^2)$,
these together with $I_1=H$, $I_2=\oo_1\oo_5-\oo_2\oo_4+\frac12\oo_3^2$ and the right-invariant linear functions
$I_3=\theta_3$, $I_4=\theta_4$, $I_5=\theta_5$ and $I_6$ form 6 involutive integrals ($C=\frac12(I_4^2+I_5^2)-I_3I_6$),
so this system is Liouville integrable.
Notice that $I_2'=\theta_1\theta_5-\theta_2\theta_4+\frac12\theta_3^2$ is also an integral, and $I_2-I_2'=I_6\cdot K$,
where $K$ is the last Killing vector field (neither $I_2'$ nor $K$ commute with $I_1,\dots,I_6$,
but they make the system super-integrable).

In coordinates: $I_3=p_3$, $I_4=p_4$, $I_5=p_5$, $I_6=p_6$ and
 \begin{multline*}
I_2=
(p_1-\tfrac12x_2p_3-x_1x_2p_4-\tfrac12x_1^2x_2p_6)(p_5+x_2p_6)\\
-(p_2+\tfrac12x_1p_3+x_1x_2p_5+\tfrac12x_1x_2^2p_6)(p_4+x_1p_6)\\
+\tfrac12(p_3+x_1p_4+x_2p_5+\tfrac12(x_1^2+x_2^2)p_6)^2.
 \end{multline*}

\subsection{\sl Non-integrability of the parabolic SR-structure}\label{S5.2}

The parabolic (2,3,5,6)-distribution is given by the structure equations:
 $$
[e_1,e_2]=e_3,\ [e_1,e_3]=e_4, \ [e_2,e_3]=e_5,\ [e_1,e_4]=e_6.
 $$
With its 11-dimensional symmetry algebra it is the maximally symmetric non-holonomic rank 2
distribution in 6D, see \cite{DZ,AK}.

Up to equivalence there is only one left-invariant SR-structure
(this follows from the fact that the Tanaka prolongation $\hat\fg$ of the Carnot algebra $\fg$
has $\hat\fg_0\subset\mathfrak{gl}(\fg_1)$ equal to the Borel subalgebra, and the corresponding
group transforms the invariant SR-structures), and it is given
by the orthonormal frame $e_1,e_2$ (the symmetry dimension of this SR-structure is 6,
and so it is not maximally symmetric). The corresponding Hamiltonian
$H=\frac12(\omega_1^2+\omega_2^2)$ has the coordinate form
 \begin{equation}\label{6Dp}
2H=\bigl(p_1-\tfrac12x_2p_3-x_1x_2p_4-\tfrac12x_1^2x_2p_6\bigr)^2+
\bigl(p_2+\tfrac12x_1p_3+x_1x_2p_5\bigr)^2.
 \end{equation}
There are two Casimir functions $\omega_5=\theta_5,\omega_6=\theta_6$, and
two additional Noether integrals $\theta_3=p_3$, $\theta_4=p_4$, that form an
involutive family $I_2=p_3,I_3=p_4,I_4=p_5,I_5=p_6$.
However no other Casimirs or commuting linear integrals exist.

In search of more complicated integrals we perform the computations for the
final (6th) integral of degree $d$ and arrive at the following result.

 \begin{theorem}\label{Thm6D}
The final integral of degree $d\le6$ for the Hamiltonian  (\ref{6Dp}) of the left-invariant SR-structure
on the parabolic (2,3,5,6)-distribution does not exist.
 \end{theorem}

\begin{proof}
First let us notice that it is enough to prove non-existence of a nontrivial integral $I_6$ of degree 6.
Indeed, if a nontrivial integral $I$ of degree $d<6$ exists, then $I\cdot p_6^{6-d}$
is a non-trivial integral of degree $6$.

Therefore we shall apply the procedure described in Section~\ref{S4} to our system with $d=6$
only\footnote{In fact, we run the test for $1\leq d\leq5$ as well, confirming the same result.}.
For sextic integrals, seven prolongations need to be performed in order to achieve equality for
the upper bound $\d_6=\d_6^{(7)}$. Our computation gives:
\smallskip

\begin{center}
\begin{tabular}{|c|c|c|c|c|c|c|c|}
\hline
\# all eqns & \# $V_6$ & \# eqns $M_\text{red}$ & \#$V_\text{red}$ & rk($M_\text{red})$ & $\d_6$ \\
\hline
28512 & 20790 & 11816 & 9155 & 9155 & 130 \\
\hline
\end{tabular}
\end{center}
\smallskip

The last number $\d_6$ coincides with the number of trivial integrals $\Lambda_6^0=130$,
and hence by the discussion in \S\ref{S4} there is no integral of degree 6, which is independent of
and commuting with $I_2,\dots,I_5$.
\end{proof}

\subsection{\sl Hyperbolic and other elliptic SR-structures in 6D}\label{S5.3}

The hyperbolic rank 2 distribution with growth vector $(2,3,5,6)$ has the following structure equations:
 $$
[e_1,e_2]=e_3,\ [e_1,e_3]=e_4, \ [e_2,e_3]=e_5,\ [e_1,e_5]=e_6,\ [e_2,e_4]=e_6.
 $$
The Hamiltonian corresponding to orthonormal frame $e_1,e_2$ is
 $$
2H=\bigl(p_1-\tfrac12x_2p_3-x_1x_2p_4-\tfrac14x_1x_2^2p_6\bigr)^2+
\bigl(p_2+\tfrac12x_1p_3+x_1x_2p_5+\tfrac14x_1^2x_2p_6\bigr)^2.
 $$
There are two Casimir functions $I_6=\omega_6$ and $C=\omega_4\omega_5-\omega_3\omega_6$.
For the Hamiltonian $H=\frac12(\omega_1^2+\omega_2^2)$ these two together with $I_1=H$ and the
integrals $I_3=\theta_3$, $I_4=\theta_4$, $I_5=\theta_5$ ($\theta_6=\omega_6$)
form 6 involutive integrals, but they are functionally dependent ($C=I_4I_5-I_3I_6$).
 % However no other Casimirs or commuting linear integrals exist.

The most general left-invariant SR-structure on both the elliptic and the hyperbolic (2,3,5,6)-distributions
can be brought into the form $2H=\omega_1^2+(a\,\omega_1+b\,\omega_2)^2$, $b\neq0$, with
the same 4 Noether integrals. However no other Casimirs or commuting linear integrals exist.

In all these cases (except the elliptic case with $a=0$, $b=1$) the system is neither SR-maximally
symmetric (the symmetry algebra has $\dim=6$), nor maximally symmetric as a distribution
(the symmetry algebra has $\dim=8$).

In all these cases the search for the final integral reduces to the same problem.
We can apply the machinery used in Theorem \ref{Thm6D}, and the computations show the
same non-existence result (in all cases except the elliptic $a=0$, $b=1$).
This non-existence of low degree integrals suggests that these Hamiltonians are not integrable.

%%%%%%%%%%%%%%%%%
\section{Maximally symmetric SR-structures in dimension 7}\label{S6}
%\section{(Non-)Integrability of SR-structures in dimension 7}\label{S6}

A non-integrability effect established in the previous section happens also in higher dimensions.
We noted that the parabolic distribution $\Delta$ in 6D is maximally symmetric, but
for the left-invariant parabolic SR-structure $(\Delta,g)$ in 6D the symmetry algebra of $(\Delta,g)$
is minimal possible: the algebra of left-translations $\fg$.

In general, the symmetry algebra of a left-invariant SR-structure $(\Delta,g)$ on a Carnot group $G$
is a graded Lie algebra $\tilde\fg$ and it contains the Lie algebra of $G$, namely
$\fg=\fg_{1}\oplus\dots\fg_{\nu}\subset\tilde\fg$. The additional part is contained at most in the zero
grading\footnote{We provide a simple proof of this fact in Appendix \ref{A:2}.}:
$\tilde\fg/\fg=\tilde\fg_0$ \cite{Mo}. Clearly this piece is at most 1-dimensional
$\tilde\fg_0\subset\mathfrak{so}(\fg_{1},g)$.

Thus $\dim\op{Sym}(\Delta,g)\le\dim\fg+1$. The equality is attained if the rotation endomorphism
$\phi\in\mathfrak{so}(\fg_{1},g)$ extends (uniquely) to a grading preserving derivation of $\fg$.
Let us investigate if such a \textit{maximally symmetric} left-invariant SR-structure on a Carnot
group is integrable.

In 6D the only maximally symmetric SR-structure is the (unique up to scale) SR-structure on the
elliptic (2,3,5,6)-distribution (with $\dim\op{Sym}=7$) and it is integrable. Consider the case $\dim G=7$.

Here the only maximally symmetric SR-structure $g$ on a rank 2 distribution $\Delta$ on a 7D Carnot group $G$
(that is not a prolongation from lower dimensions) with $\dim\op{Sym}(\Delta,g)=8$ has growth vector
(2,3,5,7) and the following structure equations\footnote{These are obtained from the (2,3,5,6) parabolic
distribution by the central extension technique of \cite{AK}.}
of the graded nilpotent Lie algebra $\fg=\op{Lie}(G)$:
 \begin{equation}\label{ell7D2}
  \begin{array}{lcr}
&[e_1,e_2]=e_3,\ [e_1,e_3]=e_4,\ [e_2,e_3]=e_5,&\\
&[e_1,e_4]=-[e_2,e_5]=e_6,\ [e_1,e_5]=[e_2,e_4]=e_7.\vphantom{\dfrac{A}{A}}&
 \end{array}
  \end{equation}
Here $H=\frac12(\oo_1^2+\oo_2^2)$ and $\fg_0=\langle e_2\ot\oo_1-e_1\ot\oo_2\rangle$.
There are 3 Casimir functions $\oo_6$, $\oo_7$ and
$\oo_3(\oo_6^2+\oo_7^2)-\frac12(\oo_4^2-\oo_5^2)\oo_6-\oo_4\oo_5\oo_7$.
The involutive family of integrals $\theta_3,\dots,\theta_7$ generates these Casimirs
and together with the Hamiltonian they lack 1 more integral for Liouville integrability.
In local coordinates, we have
 \begin{multline}\label{7De}
2H=\left(p_1-\tfrac12x_2p_3-x_1x_2p_4-\tfrac12x_1^2x_2p_6-\tfrac14x_1x_2^2p_7\right)^2\\
+\left(p_2+\tfrac12x_1p_3+x_1x_2p_5-\tfrac12x_1x_2^2p_6+\tfrac14x_1^2x_2p_7\right)^2,
 \end{multline}
and the integrals are $I_2=p_3$, \dots, $I_6=p_7$. Looking for the final integral $I_7$, we
again invoke the method of Section \ref{S4} to obtain:

 \begin{theorem}\label{Thm7D}
The final integral of degree $d\le6$ for the Hamiltonian (\ref{7De}) of the left-invariant SR-structure
on the (2,3,5,7)-distribution given by \eqref{ell7D2} does not exist.
 \end{theorem}

 \begin{proof}
We perform the same computations as in the proof of Theorem~\ref{Thm6D}.
In this case, our computer capacities allow us to study polynomial integrals up to degree $d=5$.
We need six prolongations to arrive at a definite conclusion, which is presented in the table:
\smallskip

\begin{center}
\begin{tabular}{|c|c|c|c|c|c|c|c|}
\hline
\# all eqns & \# $V_5$ & \# eqns $M_\text{red}$ & \#$V_\text{red}$ & rk($M_\text{red}$) & $\d_5$ \\
\hline
25872 & 16632 & 9397 & 6993 & 6993	& 166 \\
\hline
\end{tabular}
\end{center}
\smallskip
Since the number $\d_5=\delta_{5}^{(6)}$ coincides with the number of trivial integrals
$\Lambda^0_5=166$, we conclude absence of the final integral of degree $d\le5$.

To handle the case of degree $d=6$, we use the modular approach, described in Section \ref{S4}\!\ref{S4.3}.
The computation concludes faster, but to reach a definite answer we need a suitably large prime.
In our case $p=101$ suffices and we obtain the following result:
% The computation uses less data and concludes faster, and with $p=101$ we obtain the following result:
\smallskip

\begin{center}
\begin{tabular}{|c|c|c|c|c|c|c|c|}
\hline
\# all eqns & \# $V_6$ & \# eqns $M_\text{red}$ & \#$V_\text{red}$ & rk($M_\text{red}$) & $\d_6[101]$ \\
\hline
61776 & 41580 & 19137 & 15848 & 15848 & 296 \\
\hline
\end{tabular}
\end{center}
\smallskip
This computation implies $\d_d[p]=\Lambda_d^0$, and we conclude non-existence of the
final integral of degree $d\le6$.
 % non-existence of an (independent, commuting) last integral of the required degree.
 \end{proof}

 \begin{rk}
The indicated $p$ for $d=6$ is not claimed to be the minimal possible. But search for the minimal $p$
requires more computer time. For instance, with $d=5$ the computation for $d=5$
gives $\d_d[p]>\Lambda_d^0$ for the primes $p=2,3,\dots,29$, and we obtain equality
(implying non-existence of degree 5 integral) for the next primes $p=31$, $37$ and $41$.
 \end{rk}

%%%%%%%%%%%%%%%%%
\section{On integrability of SR-structures in dimension 8}\label{S7}
%\section{(Non-)Integrability of SR-structures in dimension 8}\label{S7}

There are two SR-structures $g$ on a rank 2 distribution $\Delta$ on a 8D Carnot group $G$
(that is not a prolongation from lower dimensions) with $\dim\op{Sym}(\Delta,g)=9$: one with the growth vector
(2,3,5,6,8) and the other with the growth vector (2,3,5,8). The distributions are obtained
by central extension from 7D as in \cite{AK},
and we take the (unique up to scale) $\mathfrak{so}(2)$-symmetric metric~$g$
(in the cases, when $\fg_0\supset\mathfrak{so}(2)$).

The second SR-structure $(\Delta,g)$ has a more symmetric underlying distribution (with the symmetry
dimension 12 vs.\ 10), but it is the first one that is integrable.

\subsection{\sl The (2,3,5,6,8) SR-structure.}\label{S7.1}

The structure equations of the algebra $\fg=\op{Lie}(G)
=\fg_1\oplus\dots\oplus\fg_5=\langle e_1,e_2\rangle\oplus\langle e_3\rangle\oplus\langle e_4,e_5\rangle
\oplus\langle e_6\rangle\oplus\langle e_7,e_8\rangle$
are the following:
 \begin{gather*}
[e_1,e_2]=e_3,\
[e_{1,2},e_3]=e_{4,5},\ [e_1,e_4]=[e_2,e_5]=e_6,\\
[e_{1,2},e_6]=e_{7,8},\ [e_3,e_{4,5}]=i\,e_{7,8},
 \end{gather*}
where we use complex notations $e_{a,b}=e_a+i\,e_b$. In this form it is obvious that the
action of $\op{SO}(2)$ on $\fg$, composed of the standard action on $\fg_1$, $\fg_3$, $\fg_5$
and the trivial action on $\fg_2$, $\fg_4$, is an automorphism.

The left-invariant Hamiltonian $H=\frac12(\oo_1^2+\oo_2^2)$ has 5 commuting right-invariant
Killing fields (integrals) $I_2=\theta_4$, $I_3=\theta_5$, $I_4=\theta_6$, $I_5=\theta_7$, $I_6=\theta_8$. In addition, there are 2 Casimir functions
 \begin{gather*}
I_7=\oo_1\oo_8-\oo_2\oo_7+\oo_3\oo_6-\frac{\oo_4^2+\oo_5^2}2=
\theta_1\theta_8-\theta_2\theta_7+\theta_3\theta_6-\frac{\theta_4^2+\theta_5^2}2,\\
C=\oo_4\oo_7+\oo_5\oo_8-\tfrac12\oo_6^2=\theta_4\theta_7+\theta_5\theta_8-\tfrac12\theta_6^2,
 \end{gather*}
of which the second is dependent on $I_2,\dots,I_7$. Yet we have one more quadratic integral
 $$
I_8=\oo_1\oo_5-\oo_2\oo_4+\tfrac12\oo_3^2,
 $$
and it is straightforward to check that the involutive integrals $I_1=H,I_2,\dots,I_8$
are functionally independent almost everywhere on $T^*G$.
Consequently, the considered SR-structure is Liouville integrable.
Notice that $I_8'=\theta_1\theta_5-\theta_2\theta_4+\tfrac12\theta_3^2$ is different from $I_8$ and
is also an integral of $H$, which again manifests super-integrability.

In coordinates, denoting $\sigma^2=x_1^2+x_2^2$, we have
 \begin{gather*}
\oo_1=p_1-\tfrac12x_2p_3-x_1x_2p_4-\tfrac12x_1^2x_2p_6
-\tfrac15(\sigma^2+x_2^2)x_3p_7+\tfrac15x_1x_2x_3p_8,\\
\oo_2=p_2+\tfrac12x_1p_3+x_1x_2p_5+\tfrac12x_1x_2^2p_6+\tfrac15x_1x_2x_3p_7
-\tfrac15(x_1^2+\sigma^2)x_3p_8,\\
\oo_3=p_3+x_1p_4+x_2p_5+\tfrac{\sigma^2}2\,p_6+(\tfrac{\sigma^2}{10}x_1+x_2x_3)p_7
+(\tfrac{\sigma^2}{10}x_2-x_1x_3)p_8,\\
\oo_4=p_4+x_1p_6+\tfrac12x_1^2p_7+(\tfrac12x_1x_2-x_3)p_8,\\
\oo_5=p_5+x_2p_6+(\tfrac12x_1x_2+x_3)p_7+\tfrac12x_2^2p_8,\\
\oo_6=p_6+x_1p_7+x_2p_8,\ \oo_7=p_7,\ \oo_8=p_8,
 \end{gather*}
and $\theta_i=p_i$ for $4\leq i\leq8$; the formulae for $I_i$ follow.

\subsection{\sl The (2,3,5,8) SR-structure.}\label{S7.2}

The free truncated graded nilpotent Lie algebra
$\fg=\fg_1\oplus\dots\oplus\fg_4=\langle e_1,e_2\rangle\oplus\langle e_3\rangle\oplus\langle e_4,e_5\rangle
\oplus\langle e_6,e_7,e_8\rangle$ with the structure equations
 \begin{gather*}
[e_1,e_2]=e_3,\ [e_1,e_3]=e_4,\ [e_2,e_3]=e_5,\\
[e_1,e_4]=e_6,\ [e_1,e_5]=[e_2,e_4]=e_7,\ [e_2,e_5]=e_8
 \end{gather*}
was also studied in \cite{Sa}. The left-invariant Hamiltonian $H=\frac12(\oo_1^2+\oo_2^2)$ has
6 commuting right-invariant Killing fields, leading to Noether integrals
$I_2=\theta_3$, $I_3=\theta_4$, $I_4=\theta_5$, $I_5=\theta_6$, $I_6=\theta_7$, $I_7=\theta_8$.
In addition, there is 1 cubic Casimir function, but it depends on the linear integrals.

Thus we again lack one final integral for integrability. To set up its computation
we write the Hamiltonian in local coordinates:
 \begin{multline}\label{8De}
2H=\left(p_1-\tfrac12x_2p_3-\tfrac12(x_1^2+x_2^2)p_5-\tfrac14x_1x_2^2p_7-\tfrac16x_2^3p_8\right)^2\\
+\left(p_2+\tfrac12x_1p_3+\tfrac12(x_1^2+x_2^2)p_4+\tfrac16x_1^3p_6+\tfrac14x_1^2x_2p_7\right)^2.
 \end{multline}

%The bounds obtained from the computations are compared with (\ref{Lambda}) for $D=8$.
%Our computer capacities allowed us to study the (2,3,5,8)-problem up to degree~5 integrals:

 \begin{theorem}\label{Thm8D}
The final integral of degree $d\le5$ for the Hamiltonian (\ref{8De}) of the left-invariant SR-structure
on the (2,3,5,8)-distribution does not exist.
 \end{theorem}

 \begin{proof}
We use again the procedure from Section~\ref{S4} to show non-existence of
a non-trivial integral of degree~$5$. After six prolongations of the PDE system,
we arrive at the following table:
\smallskip

\begin{center}
\begin{tabular}{|c|c|c|c|c|c|c|c|}
\hline
\# all eqns & \# $V_5$ & \# eqns $M_\text{red}$ & \#$V_\text{red}$ & rk($M_\text{red}$) & $\d_5$ \\
\hline
48048 & 28512 & 4439 & 3514 & 3514 & 314 \\
\hline
\end{tabular}
\end{center}
\smallskip

\noindent The upper bound $\delta_5=314=\Lambda^0_5$ realizes the number of trivial integrals,
and proves that no final (8th) integral of degree $d=5$ exists.
 \end{proof}

%%%%%%%%%%%%%%%%%
\section{Reduction to the system with 2 degrees of freedom}\label{S8}

In this section, we give a uniform description of several of the previously discussed systems in terms of first order ODE systems in 3D. In particular, we reformulate in this way the three systems exhibiting non-integrable behavior, namely the (2,3,5,6) parabolic, (2,3,5,7) elliptic and (2,3,5,8) free truncated
SR-struc\-tures given by the Hamiltonians (\ref{6Dp}), (\ref{7De}) and (\ref{8De}). In addition, the same
reduction can be performed for the general (2,3,5,6) elliptic and hyperbolic SR-structures.

First, note that in all these cases the Hamiltonian is a sum of two squares and so can be expressed as
\begin{equation}\label{eqn:reduced-hamiltonian}
 2H=\rho^2\cos^2\!z+\rho^2\sin^2\!z,
\end{equation}
and $p_i=c_i$ for $i=3,\dots,D$ are the Noether integrals.
Symplectic reduction via these integrals (fixing them and forgetting about $x_i$, $3\le i\le D$, of
which nothing depends) is a classical procedure, see \cite{Wh,A}. Thus, in view of \eqref{eqn:reduced-hamiltonian}, Hamilton's equations can be rewritten in terms of $x,y,z$ and $\rho$
(as well as $c_3,\dots,c_8$).

For instance, in the case of the parabolic (2,3,5,6)-problem, we express the coordinates $p_1,p_2$ in terms of the coordinate $z$ as follows:
\begin{align*}
p_1 &=\rho\cos z+\tfrac12x_2c_3+x_1x_2c_4+\tfrac12x_1^2x_2c_6, \\
p_2 &=\rho\sin z-\tfrac12x_1c_3-x_1x_2c_5.
\end{align*}

Next, we can confine to an energy shell, that is fix $H=\frac12\rho^2=\op{const}$, which reduces the dynamics to the manifold $S_1T^*\R^2=\R^2(x,y)\times S^1(z)$, where we let $x=x_1$, $y=x_2$.
Without loss of generality we can assume $\rho=1$.
After an appropriate change of coordinates, the Hamiltonian equation $\dot\eta=\{\eta,H\}$ on the energy shell writes as the $3\times3$ system:
 \begin{equation}\label{3s3}
\dot x=\cos z,\ \dot y=\sin z,\ \dot z=Q(x,y),
 \end{equation}
where $Q=Q(x,y)$ is a quadratic polynomial. This polynomial can be brought to the following normal form
($a\neq0$ \& $b\neq0$)
 \begin{align}
\label{Q6D}
Q&=Q_1(x,y)=a\,x^2+b\,y\ \text{ for }D=6\text{ parabolic},\\
\label{Q78D}
Q&=Q_2(x,y)=a\,x^2+b\,y^2+c\ \text{ for }D=7,8
 \end{align}
(the latter formula contains also the 6D elliptic and hyperbolic cases).
Take, for instance, the $6D$ parabolic case, formula \eqref{Q6D}. In this example,
we have $a=c_6/2$ and $b=c_5$, and we assume $a,b\not=0$.

Notice that the condition $a,b\neq0$ is important, as otherwise the system fibers over a 2D flow,
which can never be chaotic.

A similar effect happens for $a=b$ and $Q=Q_2(x,y)$, where a change of variables $x=r\cos\psi$, $y=r\sin\psi$ reduces the system to a 2D flow with coordinates $r$ and $s=z-\psi$. This latter case corresponds to the 6D elliptic maximally symmetric SR-structure. The corresponding 3D
system possesses the following integral
\[
  F=\tfrac{1}{4}a\,r^4+\tfrac{c}{2}\,r^2-r\sin(s),
\]
which corresponds exactly to the integral $I_2$ identified in Section~\ref{S5}, cf.\ also \cite{V2}.
However, for the general $a,b$, it will be shown in the next section that the system exhibits a chaotic behavior.

 \begin{rk}
One can check that in the complement to a hypersurface the following 1-form on
$\R^2(x,y)\times S^1(z)$ is contact:
 $$
\a=\tfrac13(a\,x^3dy-b\,y^3dx)+\tfrac{c}2(x\,dy-y\,dx)+\cos z\,dx+\sin z\,dy.
 $$
In this domain its Reeb vector field $R_\a$, given by the two conditions $\a(R_\a)=1$ and
$d\a(R_\a,\cdot)=0$, preserves the volume form $\a\we d\a$ and so is divergence-free with respect
to it (the Reeb field $R_\a$ plays a distinguished role in contact geometry). Our vector field,
given by (\ref{3s3}) for $Q=Q_2(x,y)$, is proportional to $R_\a$, and so has the same trajectories.
For $Q=Q_1(x,y)$ the situation is similar, if $\a$ is properly modified.
 \end{rk}

We conclude this section with a note on the resemblance of system (\ref{3s3}) to a driven pendulum
in the case $Q=Q_1(x,y)$. Let us eliminate $x,y$ from (\ref{3s3}). Differentiating $\dot{z}$ and
replacing $\dot{x}$ and $\dot{y}$ via ODE (\ref{3s3}), we get the following 3rd order ODE on $z=z(t)$, where
$\Delta=\frac{d}{dt}\circ\frac1{\cos z}$:
 $$
\Delta\,(z''-b\sin z)=2a\cos z,
 $$
which can be written in non-local form as:
\begin{equation} \label{eq:PendEq}
 z''-b\sin z=\Delta^{-1}(2a\cos z)=2a\cos z\,D_t^{-1}\cos z.
\end{equation}
In this form it resembles the driven pendulum $z''-b\sin z=a\cos kt$ without dissipation.
For $a=0$ system~\eqref{eq:PendEq} is the simple pendulum when $b<0$, while for $b>0$
the second term on the left hand side describes a repulsive
power\footnote{For instance, when $z\ll1$ the solutions are hyperbolic.}.
However, contrary to the driven pendulum, where the right hand side is an external force,  system~\eqref{eq:PendEq} seems to be self-driven.
The evolution of this system is shown in Fig.~\ref{Fig:drivenpendulum}
for three different parameter combinations.
The orbital dynamics in Fig.~\ref{Fig:drivenpendulum} is quite complex and resembles the dynamics of the damped driven pendulum (see, e.g., Fig. 9, 10 in \cite{H} and references therein),
indicating non-integrability.
This resemblance appeals for a more systematic numerical analysis of system~\eqref{3s3},
which is provided in the next section.

\begin{figure}[htp]
\centering
\includegraphics[width=0.32\textwidth]{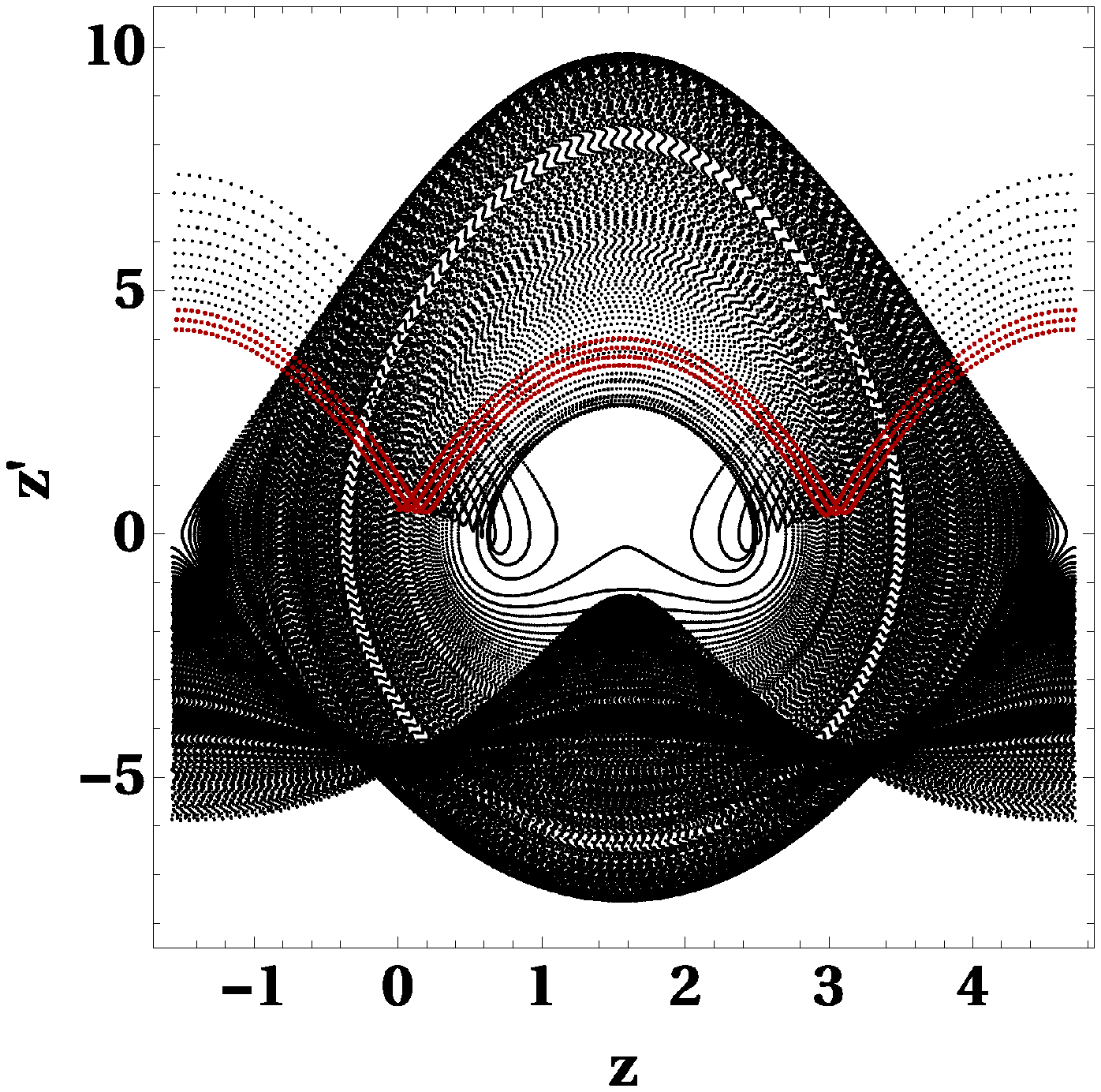}
\includegraphics[width=0.32\textwidth]{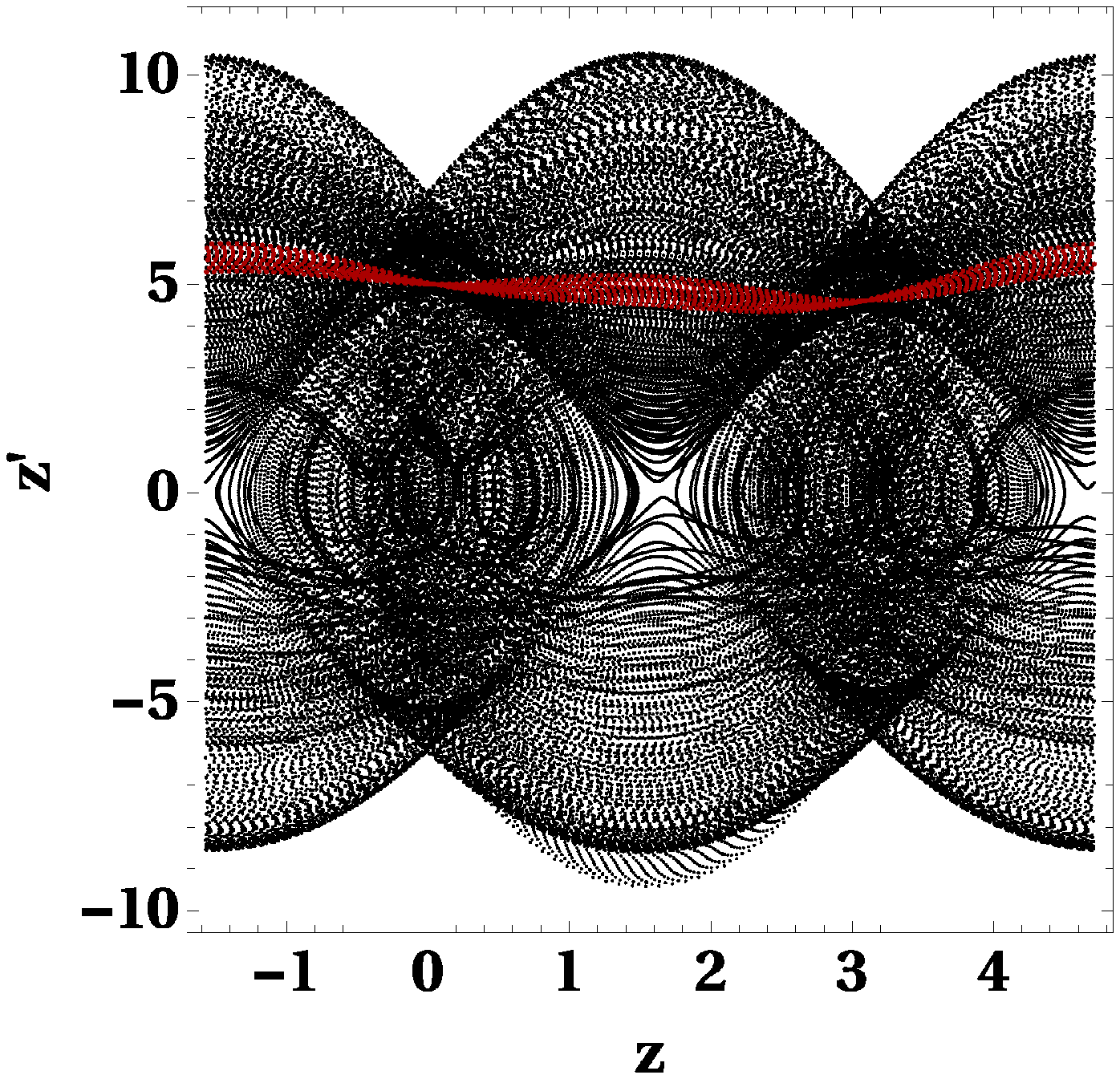}
\includegraphics[width=0.32\textwidth]{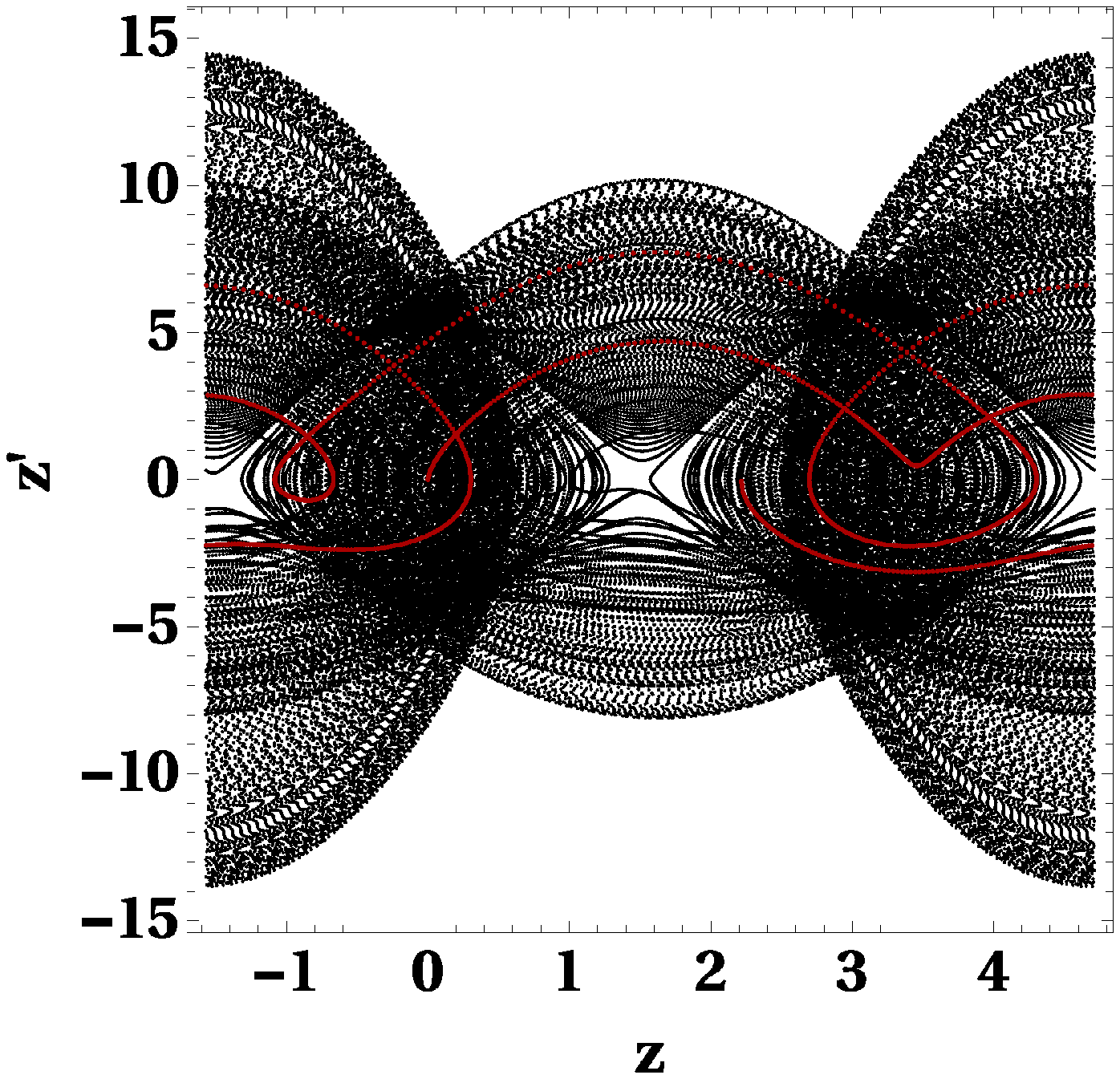}
\caption{The orbital evolution of variables $z,z^\prime$ for the $D=6$ parabolic case with the
parameters $a=10$ and $b=-0.1$ (left panel), $b=-1$ (middle panel),
$b=1$ (right panel). The initial conditions for $(x,y,z)$ are $(0,-5,0)$ in the left and middle panels
and $(0,0,0)$ in the right panel. The red curves show the evolution in the time interval $0 \leq t \leq 1$,
while the black dots continue it to time 1000.}
\label{Fig:drivenpendulum}
\end{figure}

%%%%%%%%%%%%%%%%%
\section{Numerical evidence of non-integrability}\label{S9}

In this section, we provide numerical evidence showing the non-integrability of
systems \eqref{Q6D}-\eqref{Q78D} (corresponding to SR-geodesic flows with $D=6,7,8$)
by evolving the equation of motion of the reduced system
\eqref{3s3}. In Section~\ref{S8} we have already claimed that system \eqref{Q6D} resembles
the dynamics of a driven pendulum that is chaotic.  However,
this resemblance can be a mere coincidence. In order to put forward a thorough investigation of
whether in the above systems chaos appears or not, we need a more standardized method.

One of the most classical methods for finding chaos is given by investigating
the dynamics of the return map on the surfaces of section (Poincar\'e map).
We compute this numerically by evolving the equations of motion with the
Cash-Karp-Runge-Kutta scheme. The accuracy of the numerical results is checked
by reducing the integration step size by an order of magnitude and testing whether
this reduction changes the trajectory of the orbit.

\begin{figure}[htp]
\centering
\includegraphics[width=0.35\textwidth]{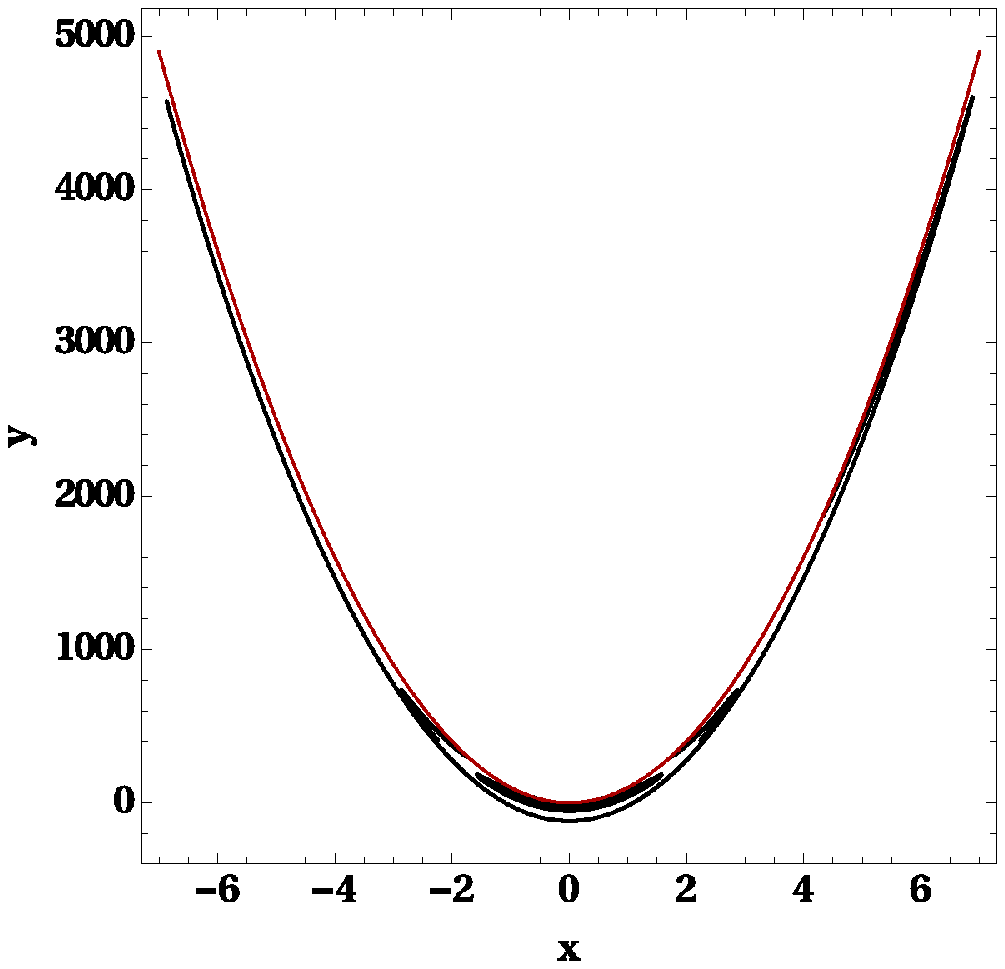}
\includegraphics[width=0.35\textwidth]{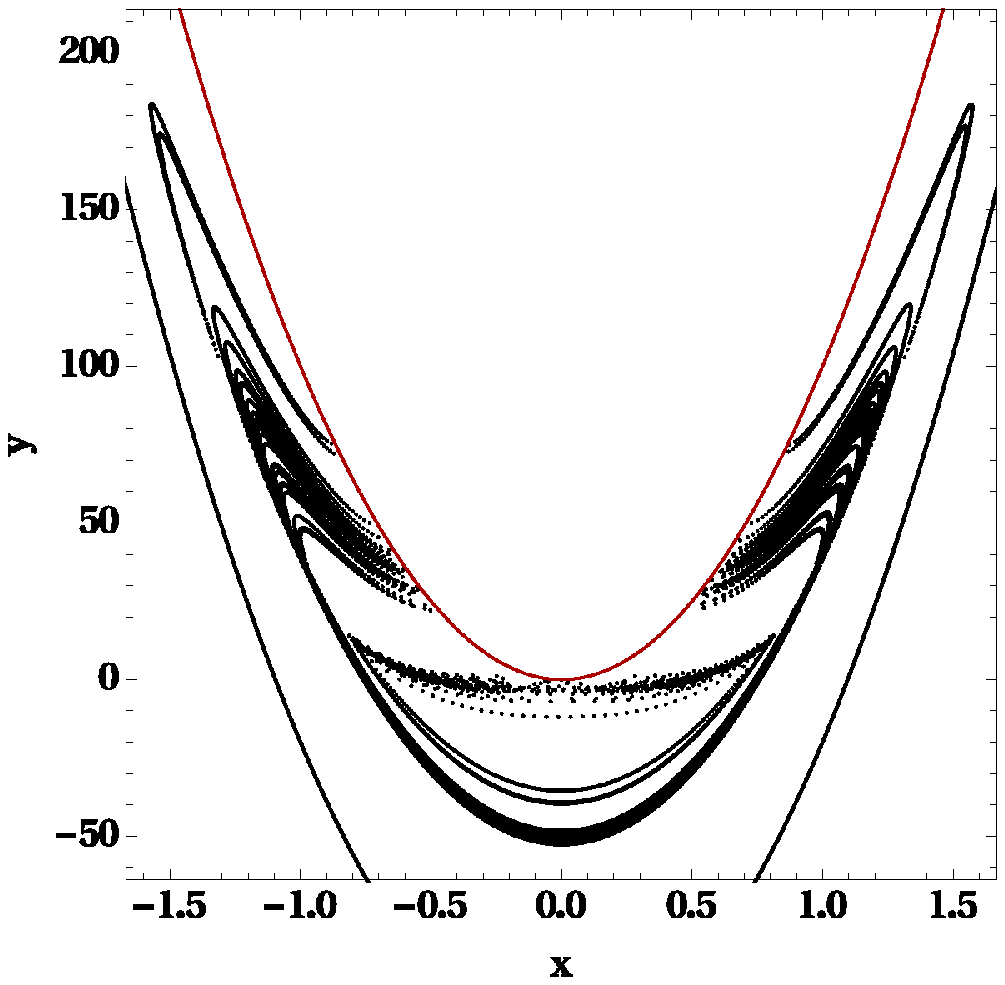}
\caption{The surface of section on $z=0$, $z^\prime>0$ (black dots) for
system \eqref{Q6D} with parameters $a=10,b=-0.1$ and the orbit starting from $x=y=0$.
The red curve $Q_1(x,y)=0$ shows the limit of this section. The left panel
shows $10^5$ sections, while the right panel zooms the left plot around
the starting point.}
\label{Fig:6Dz0}
\end{figure}

Surfaces of section were employed in \cite{Sa} for the $D=8$ case as well\footnote{Equations
(80)-(82) of \cite{Sa} with $q=0$ correspond to our  \eqref{3s3}
with $Q=Q_2(x,y)$.}. There the surface $z=0$,
$z^\prime>0$ has been chosen as the Poincar\'{e} section. However, if we employ the same
surface of section for the $D=6$ parabolic case we encounter a problem. Namely, the surface
of section $z=0$, $z^\prime>0$ does not meet all the trajectories in the phase space,
because of the parabolic form of $Q_1(x,y)$. For example, in Fig.~\ref{Fig:6Dz0}
the red curve $Q_1(x,y)=0$ sets a limit for the section we can plot, and creates
an obstacle to study the whole phase space. In other words, the surface $z=0$,
$z^\prime>0$ for system \eqref{Q6D} is not a good choice for the Poincar\'{e} section.
Moreover, the oscillations across the $x$-axis indicate that the system is non-compact.
Namely, as an orbit evolves it tends to reach larger and larger values of $|y|$
and $|x|$.

\begin{figure}[htp]
\centering
\includegraphics[width=0.45\textwidth]{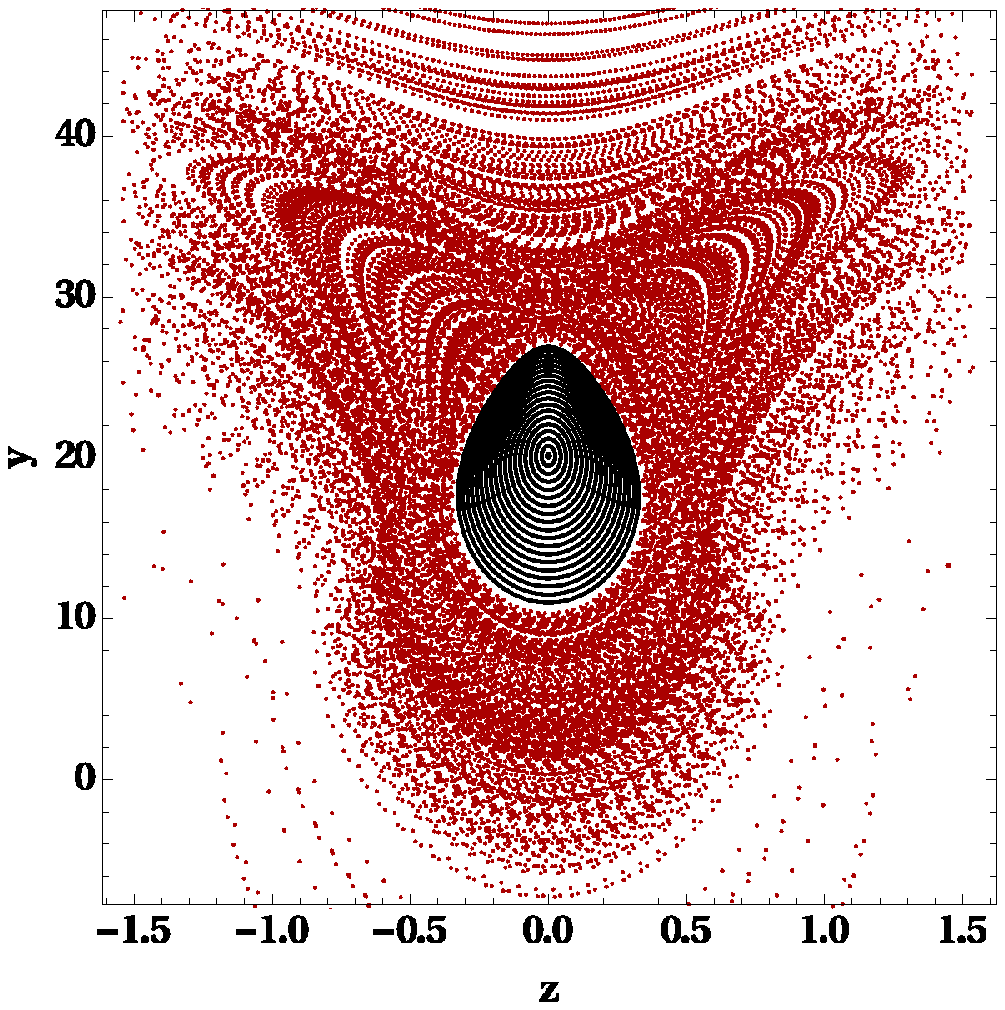}
\includegraphics[width=0.45\textwidth]{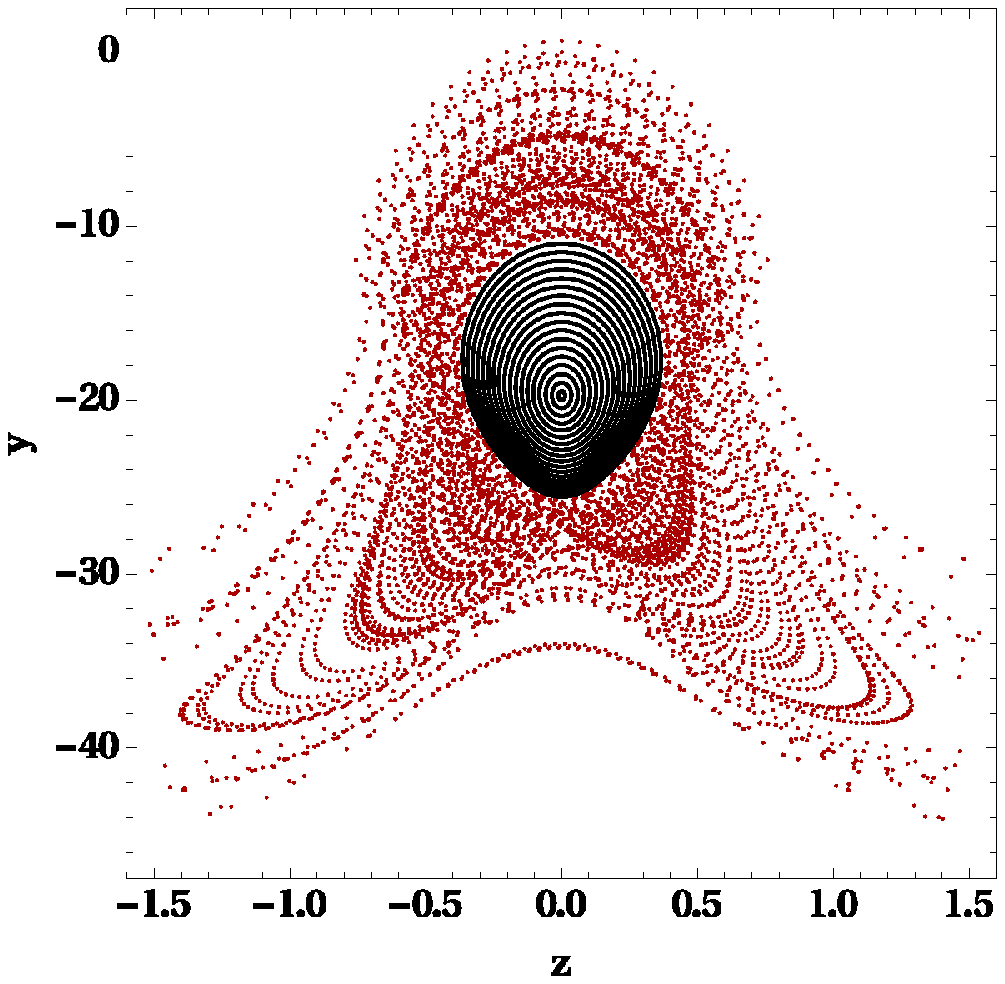}
\caption{Details of the surface of section on $x=0$, $x^\prime>0$ for system \eqref{Q6D}
with parameters $a=10$ and $b=-0.1$ (left panel), $b=0.1$ (right panel).
The black closed curves represent regular orbits, while the red dots correspond
to one irregular orbit.}
\label{Fig:6Da0}
\end{figure}

Because of the above mentioned oscillations across the $x$-axis shown in
Fig.~\ref{Fig:6Dz0}, we assumed that a good surface of section for system \eqref{Q6D}
would be the surface $x=0$, $x^\prime>0$. This assumption has proven
to be correct and we show the results on Fig.~\ref{Fig:6Da0}. In both panels of
Fig.~\ref{Fig:6Da0} we can see a region of concentric closed curves (black curves),
which represent regular orbits. The center of these regular orbits lies around the point
$(z,y)=(0,20)$ in the left panel, and around the point $(z,y)=(0,-20)$ in the right panel.
The concentric curves indicate that the central point corresponds to a stable periodic orbit.

In both cases around these concentric orbits lie an irregular orbit (red dots), which
tends to cover all the available phase space in the complement to the regular orbits.
The irregular orbit apparently stems from a point around $(z,y)=(0,30)$ in the left panel, and around
$(z,y)=(0,-30)$ in the right panel. Both these points match the appearance of unstable periodic orbits.

These irregular orbits in Fig.~\ref{Fig:6Da0} indicate that the $D=6$ parabolic system is
non-integrable. Note that the plots of Fig.~\ref{Fig:6Da0} do not show
the whole phase space, because the system is non-compact. Instead we focus our
plots on the region around the regular orbits and near the unstable point, where
the irregular features are more prominent.

\begin{figure}[htp]
\centering
\includegraphics[width=0.7\textwidth]{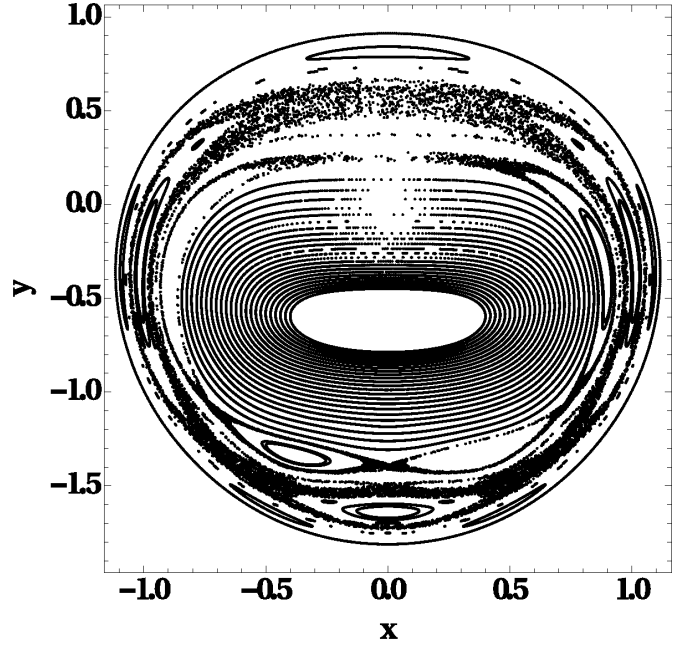}
\caption{Dynamics on the surface of section $z=0$, $z^\prime>0$ for system \eqref{Q78D} with
parameters $a=2$, $b=1$, $c=0$ (elliptic case).}
\label{Fig:7Dz0el}
\end{figure}

System \eqref{Q78D} can be separated in two categories: the elliptic ones
($ab>0$), and the hyperbolic ones ($ab<0$). In the elliptic case the surface of
section in Fig.~\ref{Fig:7Dz0el} tells straightforwardly that the system is
non-integrable. Namely, in Fig.~\ref{Fig:7Dz0el} we can discern the characteristic features of a
non-integrable system like chaotic regions and islands of stability belonging to
Birkhoff chains. In the particular case of the system corresponding to $D=8$, the indicated
non-integrability is in agreement with the non-integrability conjecture of \cite{Sa}.

\begin{figure}[htp]
\centering
\includegraphics[width=0.30\textwidth]{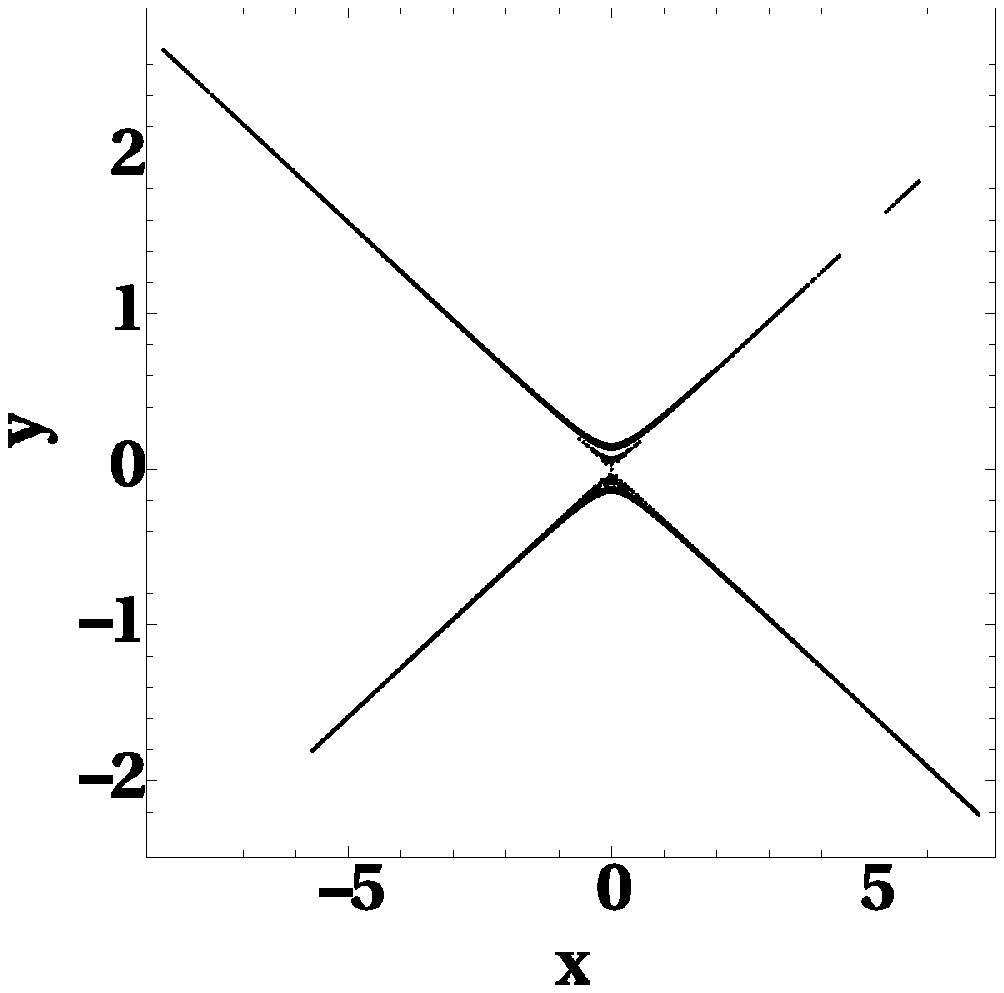}
\includegraphics[width=0.30\textwidth]{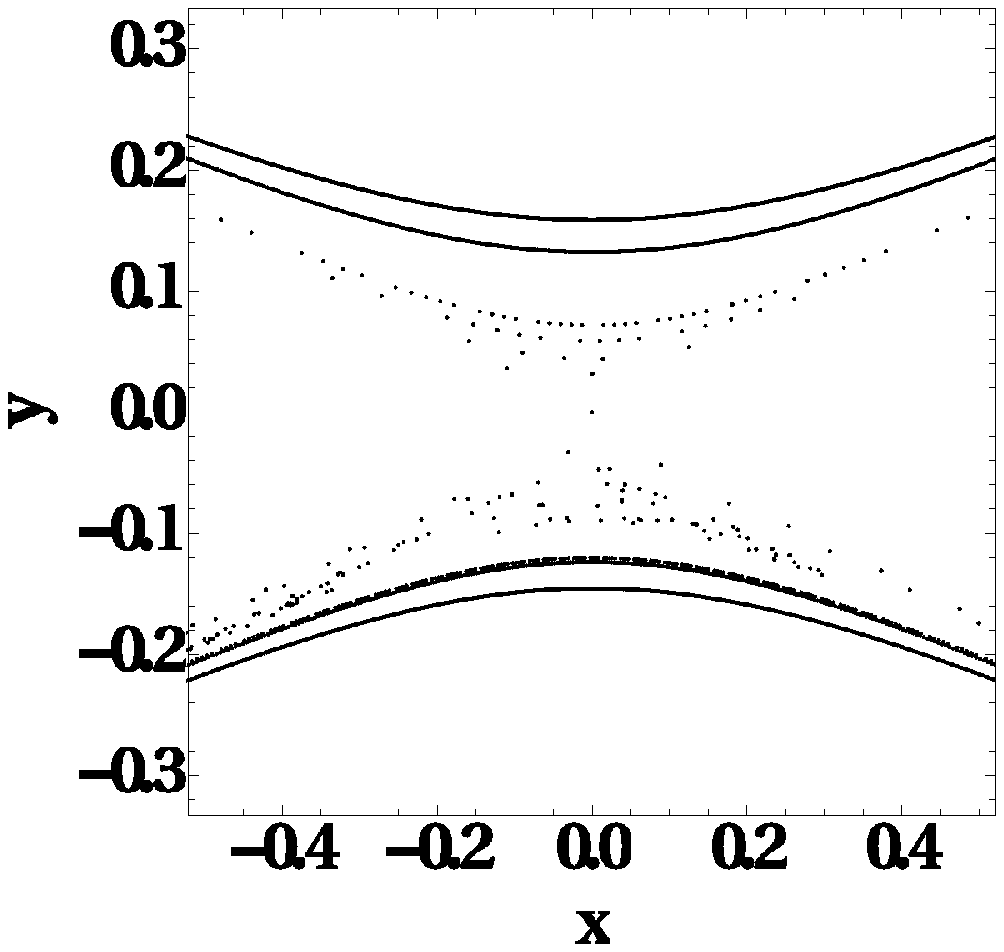}
\includegraphics[width=0.30\textwidth]{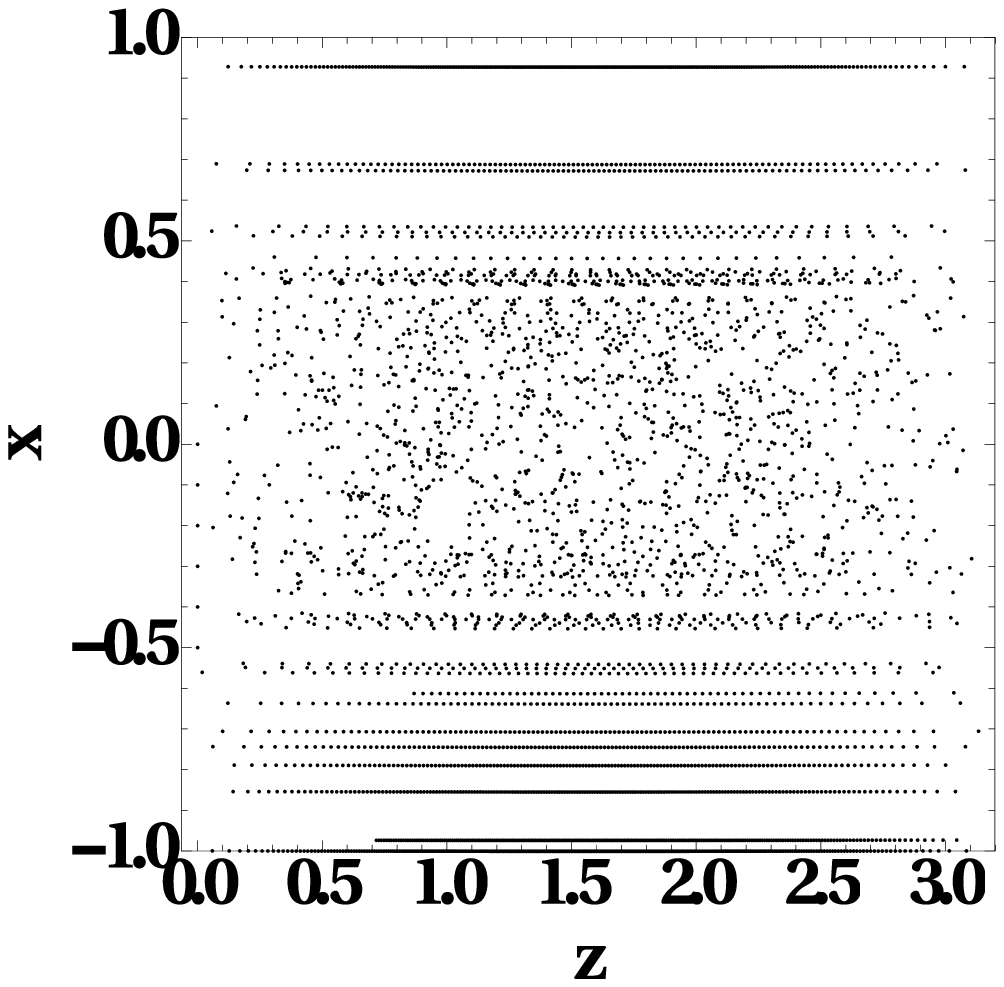}
\caption{The surface of section on $z=0$, $z^\prime>0$ (left and middle panel)
for system \eqref{Q78D} with parameters $a=-10^3,~b=10^4$. The orbit starts from
$x=y=0$, and evolves for $10^4$ sections. The whole orbit is shown in the left
panel, while a detail focusing on the starting point is shown in the middle
panel. The right panel shows the same orbit on the surface $y=0$, $y^\prime>0$.}
\label{Fig:7Dhyp}
\end{figure}

We can assert non-integrability also for the hyperbolic case on the ground of analytic dependence
on the parameters $a,b$ of our system (assuming the integrals should share the same property).
However, we can confirm this numerically as well, and we do it in Fig.~\ref{Fig:7Dhyp}:
the hyperbolic orbit is shown on two different surfaces
of section, and both of these surfaces indicate that the orbit is irregular,
and therefore, the systems \eqref{Q78D}, corresponding to $D=7,~8$, are non-integrable.

\bigskip

\appendix

\section{SR-structures on 3D Lie groups}\label{A:1}

Every left-invariant SR structure on a 3-dimensional Lie group $G$
is determined by a 2-dimensional subspace (not subalgebra) of the Lie algebra $\fg$ and a metric on it.
The classification of such is due to \cite{VG}, and this reference also contains the integration of the
equations of geodesics in terms of a semi-direct product.

Liouville integrability of left-invariant SR structures on 3D Lie groups $G$
was proven in the preprint arXiv:math/0105128 of \cite{Kr}. % (omitted in the paper version)
It was later re-visited in \cite{MS}. We provide a short proof here for completeness.

 \begin{theorem}\label{Thm3D:appendix}
Non-holonomic geodesic flows of left-invariant SR-metrics on 3-dim
Lie groups are Liouville integrable with polynomial integrals.
 \end{theorem}

 \begin{proof}
The left-invariant Hamiltonian $2H=\oo_1^2+\oo_2^2$ commutes with all right-invariant
forms $\theta_i$. Every 3-dimensional Lie algebra $\fg$ has a Casimir function $C\in C^\infty(\fg^*)$
(because $G$ has odd dimension), so the involutive set of integrals is: $I_2=C$, $I_3=\theta_i$,
where the number $i$ is chosen such that $I_1=H,I_2,I_3$ are functionally independent.

Moreover this $C$ is linear for the Heisenberg algebra and quadratic for simple Lie algebras
$\mathfrak{sl}(2)$, $\mathfrak{so}(3)$, but it can be non-algebraic (depending on parameters)
in the remaining semi-direct cases $\fg=\R^1\ltimes\R^2$. In these cases,
$\mathfrak{h}=\R^2$ is an Abelian subalgebra.
The right-invariant forms $I_2,I_3$ associated to a basis in $\mathfrak{h}$ are integrals in involution.
The Hamiltonian $H$ is algebraically (and functionally) independent of those,
because otherwise it would be bi-invariant. This completes the proof.
 \end{proof}

\section{Prolongation of Killing symmetries}\label{A:2}

Let $\fg=\fg_{-\nu}\oplus\dots\oplus\fg_{-1}$ be a (finite-dimensional) graded nilpotent Lie
algebra\footnote{It is customary in Tanaka theory to use negative gradation in the basic part, so we switch
here from the notations used in the main body of the paper.}, such that
$\fg_{-1}$ generates $\fg$. The Tanaka prolongation is a graded Lie algebra $\hat\fg$ such that
$\hat\fg_{-}=\oplus_{i<0}\hat\fg_i=\fg$ and it is the maximal graded Lie algebra with this property
(its construction is outlined below). In particular, $\hat\fg_0=\mathfrak{der}_0(\fg)$ is the algebra of
grading preserving derivations of the Lie algebra $\fg$.

Given a subalgebra $\fg_0\subset\hat\fg_0$, the Tanaka prolongation $\op{pr}(\fg,\fg_0)=
\fg_{-\nu}\oplus\dots\oplus\fg_{-1}\oplus\fg_0\oplus\hat\fg_1\oplus\dots$
is naturally defined if we restrict the non-positive part to $\fg\oplus\fg_0$.
Constructively $\hat\fg_1$ consists of the homomorphisms $\vp:\fg_{-i}\to\fg_{1-i}$, $i>0$,
satisfying the Leibniz rule $\vp([x,y])=[\vp(x),y]+[x,\vp(y)]$, then we similarly define $\hat\fg_2$ etc.
If some $\hat\fg_r=0$, then also $\hat\fg_i=0$ for $i>r$ and the algebra $\hat\fg$ is finite-dimensional.

An example of reduction of $\fg_0$ is given by a left-invariant SR-structure.

 \begin{theorem}
Let $g$ be a Riemannian metric on $\fg_{-1}$ and
$\fg_0=\mathfrak{der}_0(\fg)\cap\mathfrak{so}(\fg_{-1},g)$.
Then $\op{pr}_+(\fg,\fg_0)=0$, i.e. $\hat\fg_i=0$ $\forall i>0$.
 \end{theorem}

This theorem is due to Morimoto \cite{Mo}. His proof is based on
a result due to Yatsui. In the case of our interest we give a simpler argument.

\smallskip

 {\bf Proof in the case $\dim\fg_{-1}=2$.}
Clearly the only possibility for non-zero $\fg_0$ is $\R=\mathfrak{so}(2)=\langle e_0\rangle$
that acts on $\fg_{-1}=\langle e'_{-1},e''_{-1}\rangle$ as a complex structure:
$[e_0,e'_{-1}]=e''_{-1}$, $[e_0,e''_{-1}]=-e'_{-1}$.

For $0\neq\vp\in\hat\fg_1$ there is a basis of $\fg_{-1}$ such that
$\vp(e'_{-1})=e_0$, $\vp(e''_{-1})=0$. Then for $e_{-2}=[e'_{-1},e''_{-1}]$ we have $\vp(e_{-2})=-e'_{-1}$.
Let $\tilde e_{-2}=e_{-2}$ and define recursively $\tilde e_{-s}=[e''_{-1},\tilde e_{1-s}]$, $s>2$.
We have $\vp(\tilde e_{-s})=\tilde e_{1-s}$ so by induction $\tilde e_{-s}\neq0$ $\forall s>2$,
implying that $\dim\fg=\infty$. This is a contradiction. \qed

%%%%%%%%%%%%%%%%%%%%%%

\end{document}